\documentclass[11pt,letterpaper]{amsart}
\usepackage{tikz}
\usetikzlibrary{intersections, calc, arrows.meta}
\usepackage[utf8]{inputenc}
\usepackage{amsmath} 
\usepackage{amsmath}
\usepackage{amssymb} 
\usepackage{amsfonts} 
 \usepackage{amsthm}
 \usepackage{mathtools}
 \usepackage{latexsym}
 \usepackage{amscd}
 \usepackage[all]{xy}
 \usepackage{mathrsfs}
 \usepackage{yhmath}
\newtheorem{dfn}{Definition}[section]
 \newtheorem{them}[dfn]{Theorem}
 \newtheorem{lem}[dfn]{Lemma}
 \newtheorem{prp}[dfn]{Proposition}

\newtheorem{cor}[dfn]{Corollary}
 \newtheorem{cla}[dfn]{Claim}
 \newtheorem{Que}[dfn]{Question}
 
  \theoremstyle{definition}
\newtheorem{rem}[dfn]{Remark}
\def\RR{\mathbb{R}}
\def\ZZ{\mathbb{Z}}
\def\CC{\mathbb{C}}
\def\DD{\mathbb{D}}

\def\QQ{\mathbb{Q}}
\def\ind{\mathrm{ind}}
\def\disk{\mathrm{disk}}

 \usepackage{scalerel,stackengine}
\stackMath
\newcommand\reallywidecheck[1]{%
\savestack{\tmpbox}{\stretchto{%
  \scaleto{%
    \scalerel*[\widthof{\ensuremath{#1}}]{\kern-.6pt\bigwedge\kern-.6pt}%
    {\rule[-\textheight/2]{1ex}{\textheight}}
  }{\textheight}%
}{0.5ex}}%
\stackon[1pt]{#1}{\scalebox{-1}{\tmpbox}}%
}
\begin{document}

\pagestyle{plain}
\thispagestyle{plain}

\title[]{Elliptic bindings and the first ECH spectrum for convex Reeb flows on lens spaces}
\author[Taisuke SHIBATA]{Taisuke SHIBATA}
\address{Research Institute for Mathematical Sciences, Kyoto University, Kyoto 606-8502,
JAPAN.}
\email{shibata@kurims.kyoto-u.ac.jp}

\date{\today}

\begin{abstract}
In this paper, at first we introduce a sufficient condition for a rational unknotted Reeb orbit $\gamma$  in a lens space to be elliptic by using the rational self-linking number $sl_{\xi}^{\mathbb{Q}}(\gamma)$ and the Conley-Zehnder index $\mu_{\mathrm{disk}}(\gamma^{p})$,  where $\mu_{\mathrm{disk}}$ is the Conley-Zehnder index with respect to a trivialization induced by a binding disk. As a consequence,  we show that a periodic orbit $\gamma$ in dynamically convex $L(p,1)$ must be elliptic if $\gamma^{p}$ binds a Birkhoff section of disk type and has $\mu_{\mathrm{disk}}(\gamma^{p})=3$. It was proven  in \cite{Sch} that such an orbit always exists in a dynamically convex $L(p,1)$. Next, we estimate the first ECH spectrum on dynamically convex $L(3,1)$. In particular, we show that the first ECH spectrum on  a 
 strictly convex (or non-degenerate dynamically convex) $(L(3,1),\lambda)$ is equal to the infimum of  contact areas of  certain  Birkhoff sections of disk type.
 The key of the argument is to conduct technical computations regarding indices present in ECH and to  observe the topological properties of rational open book decompositions supporting $(L(3,1),\xi_{\mathrm{std}})$ 
coming from $J$-holomorphic curves.
\end{abstract}

\maketitle
\tableofcontents
\section{Introduction and main results}
\subsection{Introduction}
Let $(Y,\lambda)$ be a oriented contact 3-manifold with $\lambda \wedge d\lambda>0$. Let $X_{\lambda}$ be the Reeb vector field. A periodic orbit is a map $\gamma:\mathbb{R}/T_{\gamma}\mathbb{Z}:\to Y$ with $\Dot{\gamma}=X_{\lambda}\circ \gamma$ for some $T_{\gamma}>0$ where $T_{\gamma}$ is the period of $\gamma$ and we write $\gamma^{p}$ for $p\in \mathbb{Z}$ as a periodic orbit of composing $\gamma$ with  the natural projection $\mathbb{R}/pT_{\gamma}\mathbb{Z}\to \mathbb{R}/T_{\gamma}\mathbb{Z}$.
Consider a periodic orbit $\gamma$. If the eigenvalues of the return map $d\phi^{T_{\gamma}}|_{\xi}:\xi_{\gamma(0)}\to\xi_{\gamma(0)}$ are positive (resp. negative) real, $\gamma$ is called positive (resp. negative) hyperbolic. If the eigenvalues of the return map $d\phi^{T_{\gamma}}|_{\xi}:\xi_{\gamma(0)}\to\xi_{\gamma(0)}$ are on the unit circle in $\mathbb{C}$, $\gamma$ is called elliptic. 

Whether there exists an elliptic periodic orbit on a given contact manifold has been studied. In particular, it is a long-standing conjecture that a convex energy hypersurface in the standard symplectic Euclidean space carries an elliptic orbit and there are many previous studies under some additional assumptions (c.f. \cite{AbMa1}, \cite{AbMa2}, \cite{DDE}, \cite{HuWa}, \cite{LoZ}). As one of them, it is natural to consider the quotient of a convex energy hypersurface by a cyclic group action which becomes a lens space  (\cite{AbMa1}, \cite{AbMa2}) . In this case, the result depends on the lens space.

Our first result is to introduce a sufficient condition for a given 
rational unknotted Reeb orbit $\gamma$ in $L(p,q)$ to be elliptic by using the rational self-linking number $sl_{\xi}^{\mathbb{Q}}(\gamma)$ and the Conley-Zehnder index $\mu_{\mathrm{disk}}(\gamma^{p})$. 

The sufficient condition is especially useful under dynamical convexity which is the notion introduced in \cite{HWZ2}. Recall that a contact 3-manifold $(Y,\lambda)$ with $c_{1}(\xi)|_{\pi_{2}(Y)}=0$ is dynamically convex if any contractible periodic orbit 
$\gamma$ satisfies $\mu_{\mathrm{disk}}(\gamma) \geq 3$. 
In particular, for any  strictly convex domain  with smooth boundary $0\in S \subset \RR^{4}=\CC^{2}$,  the contact $3$-sphere $(\partial S, \lambda_{0}|_{\partial S})$ is dynamically convex where $\lambda_{0}=\frac{i}{2}\sum_{1\leq i \leq n}(z_{i}d\Bar{z_{i}}-\Bar{z_{i}}dz_{i})$. In addition, if $(Y,\lambda)$ is  a dynamically convex contact 3-sphere, the contact structure must be tight. Since dynamical convexity is preserved under taking a finite cover,  the contact structure of a dynamically convex contact lens space must be universally tight.  See \cite{HWZ1,HWZ2}.

Besides the existence of elliptic orbits, there is a problem of the existence of Birkhoff sections of disk type. 
 A Birkhoff section of disk type for $X_{\lambda}$ on $(Y,\lambda)$ is a compact immersed disk $u:\DD \to Y$ satisfying (1). $u(\DD\backslash \partial \DD)\subset Y\backslash u(\partial \DD)$ is embedded, 
 (2). $X_{\lambda}$ is transversal to $u(\DD\backslash \partial \DD)$,
(3). $u(\partial \DD)$ is tangent to a periodic orbit of $X_{\lambda}$,
(4). for every $x\in Y\backslash u(\partial \DD)$, there are $-\infty<t_{x}^{-}<0<t_{x}^{+}<+\infty$ such that $\phi^{t^{\pm}_{x}}(x)\in u(\DD)$ where $\phi^{t}$ is the flow of $X_{\lambda}$.

Whether a dynamically convex lens space $(L(p,q),\lambda)$ has a Birkhoff section of disk type is partially known.  First of all, it was proved by Hofer Wysocki Zehnder \cite{HWZ2} that any dynamically convex $(S^{3},\lambda)$ admits a periodic orbit $\gamma$ such that $\gamma$ binds a a Birkhoff section of disk type and $\mu_{\mathrm{disk}}(\gamma)=3$.  Recently Hryniewicz and Salomão \cite{HrS} showed that  any dynamically convex $(L(2,1),\lambda)$ admits a periodic orbit $\gamma$ which binds a a Birkhoff section of disk type and $\mu_{\mathrm{disk}}(\gamma^{2})=3$.
After that  Schneider \cite{Sch} generalized it to $(L(p,1).\xi_{\mathrm{std}})$. That is, he showed that any dynamically convex $(L(p,1),\lambda)$ with $\lambda\wedge d\lambda > 0$ admits a periodic orbit $\gamma$ such that $\gamma^p$ binds a Birkhoff section of disk type and $\mu_{\mathrm{disk}}(\gamma^{p})=3$. On the other hand, The author \cite{Shi} showed by using Embedded contact homology that non-degenerate dynamically convex $(L(p,p-1),\lambda)$ with $\lambda\wedge d\lambda>0$ must have an orbit $\gamma$ binding a Birkhoff section of disk type.

It is natural to ask whether a periodic orbit binding a Birkhoff section is elliptic.  Our sufficient condition leads to the ellipticity of the binding orbit of a  Birkhoff section of disk type on $L(p,1)$ with $\mu_{\mathrm{disk}}(\gamma^{p})=3$ which was found in  \cite{Sch}.

To explain results, we recall some notions. Let $\DD$ denote the unit closed disk.

\begin{dfn}\label{unknot}
    A knot $K\subset Y^3$ is called $p$-unknotted if there exists an immersion $u:\mathbb{D}\to Y$ such that $u(\ring{\mathbb{D}})\subset Y\backslash u(\partial\mathbb{D})$ is embedded and $u|_{\partial \mathbb{D}}:\partial \mathbb{D}\to K$ is a $p$-covering map.
    \end{dfn}
\begin{rem}\label{heegaard}
   Let  $K\subset Y$ be a $p$-unknotted knot and $u:\mathbb{D}\to Y$ be a immersed disk as above. Then the union of a neighborhood  of $u(\mathbb{D})$ and $K$ is diffeomorohic to $L(p,k)\backslash B^{3}$ for some $k$ where $B^{3}$ is the 3-ball. Therefore, $Y$ is $L(p,k)\# M$ for some $k$ and a closed 3-manifold $M$. In particular, if a lens space admits a  a $p$-unknotted knot, then it is diffeomorphic to $L(p,k)$ for some $k$ \cite[cf. Section 5]{BE}.
\end{rem}
 \begin{dfn}\cite[cf. Subsection 1.1]{BE}
   Let $(Y,\lambda)$ be a contact 3-manifold with $\mathrm{Ker}\lambda=\xi$. Assume that a knot $K\subset Y$ is $p$-unknotted, transversal to $\xi$ and oriented by the co-orientation of $\xi$. Let $u:\mathbb{D}\to Y$ be  an immersion  such that $u|_{\mathrm{int}(\mathbb{D})}$ is embedded and $u|_{\partial \mathbb{D}}:\partial \mathbb{D}\to K$ is a orientation preserving $p$-covering map. Take a non-vanishing section $Z:\mathbb{D}\to u^{*}\xi$ and consider the immersion $\gamma_{\epsilon}:t\in \mathbb{R}/\mathbb{Z} \to \mathrm{exp}_{u(e^{2\pi i t})}(\epsilon Z(u(e^{2\pi i t})))\in Y\backslash K$ for small $\epsilon>0$. 
    
    Define the rational self-linking number $sl_{\xi}^{\mathbb{Q}}(K,u)\in \mathbb{Q}$ as
    \begin{equation*}
        sl_{\xi}^{\mathbb{Q}}(K,u)=\frac{1}{p^{2}} \#(\mathrm{Im}\gamma_{\epsilon}\cap u(\DD))
    \end{equation*}
    where $\#$ counts the intersection number algebraically.
    If  $c_{1}(\xi)|_{\pi_{2}(Y)}=0$, $sl_{\xi}^{\mathbb{Q}}(K,u)$ is independent of $u$. Hence  we write $sl_{\xi}^{\mathbb{Q}}(K)$.
\end{dfn}
\begin{rem}
In generall, (rational) self-linking number is defined for rationally null-homologous knot by using a (rational) Seifert surface. See \cite{BE}.
\end{rem}

 Let $(Y,\lambda)$ be a contact 3-manifold with $c_{1}(\xi)|_{\pi_{2}(Y)}=0$. Assume that a Reeb orbit $\gamma:\RR/T_{\gamma}\ZZ \to  Y$ is contractible. Take a map  $u:\mathbb{D}\to Y$  so that $u(e^{2\pi t})=\gamma(T_{\gamma}t)$. Let $\tau_{\mathrm{disk}}:\gamma^{*}\xi\to \RR/T_{\gamma}\ZZ \times \RR^{2}$ denote a symplectic  trivialization which extends to a trivialization over $u^{*}\xi$.  We  write the Conley-Zehnder index (see the next section) $\mu_{\mathrm{disk}}(\gamma)$ as $\mu_{\tau_{\mathrm{disk}}}(\gamma)$  if there is no confusion.

\subsection{Ellipticity of an unknotted orbit}
We fix an orientation on a lens space $L(p,q)$ as follows.
Let  $p\geq q>0$  be mutually prime. 
Consider the unit  4-ball $ B^{4}(1)\in \RR^{3}$. Then the boundary $\partial B^{4}(1)$ has an orientation induced by  $ B^{4}(1)\in \RR^{3}$.   The action $(z_{1},z_{2})\mapsto (e^{\frac{2\pi i}{p}}z_{1},e^{\frac{2\pi iq}{p}}z_{2})$ preserves $\partial B(1)$. Hence we have   $L(p,q)$ as  the quotient space. From now on, we assume that  $L(p,q)$ is oriented by $\partial B(1)$.

Our first result is as follows.

\begin{them}\label{thm:main}
Let $p>q>0$ be mutually prime.
    Let $\lambda$ be a contact form on $L(p,q)$ with $\lambda \wedge d\lambda>0$. Let $\mathrm{Ker}\lambda=\xi$ and $\gamma $ be a $p$-unknotted Reeb orbit in $(L(p,q),\lambda)$. Suppose that   $-2r-2p\cdot sl_{\xi}^{\mathbb{Q}}(\gamma)-\mu_{\mathrm{disk}}(\gamma^{p})$ is not divisible by $p$ for any $r \in \ZZ$ satisfying either $r=-q\,\,\mathrm{mod}\,\,p$ or    $rq=-1\,\,\mathrm{mod}\,\,p$. Then   $\gamma$ is elliptic.
\end{them}

We note that it is obvious  that a periodic orbit with $\mu_{\mathrm{disk}}(\gamma^{2})=3$ on dynamically convex $(L(2,1),\lambda)$ is elliptic as follows (see the next section for the notations in the following). As mentioned, $\mathrm{Ker}\lambda=\xi$ is universally tight if $(L(2,1),\lambda)$ is dynamically convex. Therefore $\xi$ is topologically trivial and we can take a global symplectic trivialization $\tau_{\mathrm{glob}}:\xi \to L(2,1)\times \RR^{2}$.  If a periodic orbit $\gamma$ is hyperbolic, we have $\mu_{\tau_{\mathrm{glob}}}(\gamma^2)=2\mu_{\tau_{\mathrm{glob}}}(\gamma)$ (c.f. Proposition \ref{conleybasic}). On the other hand, since $\gamma^{2}$ is contractible, it follows from the definition that $\mu_{\tau_\mathrm{glob}}(\gamma^2)=\mu_{\mathrm{disk}}(\gamma^2)$. This means that $2\mu_{\tau_{\mathrm{glob}}}(\gamma)=\mu_{\tau_\mathrm{glob}}(\gamma^2)=\mu_{\mathrm{disk}}(\gamma^2)=3$. This contradicts $\mu_{\tau_{\mathrm{glob}}}(\gamma)\in \ZZ$.

In general, the above argument can not be applied to $(L(p,1),\lambda)$ with $\lambda\wedge d\lambda>0$ unlike $L(2,1)$ because the universally tight contact structure on $L(p,1)$ is not topologically trivial for $p>2$. But it follows from Theorem \ref{thm:main} that the same result holds for $L(p,1)$ as follows.

\begin{cor}\label{main;cor}
    Let $(L(p,1),\lambda)$ with $\lambda\wedge d\lambda>0$ be dynamically convex. Let $\gamma $ be a periodic orbit such that $\gamma^p$  binds a Birkhoff section of disk type and $\mu_{\mathrm{disk}}(\gamma^{p})=3$. Then $\gamma$ is elliptic. Note that according to \cite{Sch}, such a periodic orbit always exists.
\end{cor}
\begin{rem}
    It is proved in \cite{AbMa1} that a dynamically convex $(L(p,1),\lambda)$ admits an elliptic orbit. 
\end{rem}

To apply Theorem \ref{thm:main}, the following theorem proved by Hryniewicz and Salom\~{a}o is important.

\begin{them}\cite{HrS}\label{fundament}
Let $\lambda$ be a dynamically convex contact form on $L(p,q)$. Then for a simple orbit $\gamma$,
$\gamma$ is
$p$-unknotted  and $sl_{\xi}^{\mathbb{Q}}(\gamma)=-\frac{1}{p}$ if and only if  $\gamma^{p}$ bound  
a disk which is a Birkhoff section. Moreover, this
Birkhoff section is a page of a rational open book decomposition of $L(p,q)$  such that all pages are Birkhoff sections. 
\end{them}

\begin{proof}[\bf Proof of Corollary \ref{main;cor}]
    We apply Theorem \ref{thm:main} to $L(p,1)$. In this case, $q=1$ and hence it suffices to consider $r\in \ZZ$ satisfying $r=-1\,\,\mathrm{mod}\,\,p$. Since $r=-1\,\,\mathrm{mod}\,\,p$, $\mu_{\mathrm{disk}}(\gamma^{p})=3$ and  $sl_{\xi}^{\mathbb{Q}}(\gamma)=-\frac{1}{p}$ (Theorem \ref{fundament}),  we have  $-2r-2p\cdot sl_{\xi}^{\mathbb{Q}}(\gamma)-\mu_{\mathrm{disk}}(\gamma^{p})=1 \,\mathrm{mod}\,\,p$. This means that $\gamma$ is elliptic.
\end{proof}

As a generalization, it is a natural to ask the following question.

\begin{Que}
    Let $(L(p,q),\lambda)$ (including $S^{3}$ as $p=1,\,q=0$) be dynamically convex. Does there always exist a periodic orbit $\gamma$ such that $\gamma^{p}$ binds a Birkhoff section of disk type?
\end{Que}

Note that the author does not know whether the periodic orbits are elliptic obtained in \cite{Shi} for dynamically convex $(L(p,p-1),\lambda)$ with $\lambda\wedge d\lambda>0$.

\subsection{The first ECH spectrum and Birkhoff section}
Next, we focus on the relationship between Birkhoff sections of disk type and the first ECH spectrum.

For $(L(p,q),\lambda)$, define 
\begin{equation*}
    \mathcal{S}_{p}(L(p,q),\lambda):=\{\gamma \mathrm{\,\,simple\,\,orbit\,\,of}\,\,(L(p,q),\lambda)|\,\,p\mathrm{-unknotted},\,\,sl_{\xi}^{\mathbb{Q}}(\gamma)=-\frac{1}{p}\,\}.
\end{equation*}
We  write $\mathcal{S}_{p}$ instead of $ \mathcal{S}_{p}(L(p,q),\lambda)$ if there is no confusion. 

As mentioned in Theomre \ref{fundament}, on dynamically convex $(L(p,q),\lambda)$ a periodic orbit $\gamma^p$ binds a Birkhoff section of disk type if and only if $\gamma \in  \mathcal{S}_{p}(L(p,q),\lambda)$.

Recently, it has been gradually understood that there is a connection between the Birkhoff section and the ECH spectrum.

For instance, according to \cite{HrHuRa} if $(S^{3},\lambda)$ is dynamically convex, the first ECH spectrum $c_{1}^{\mathrm{ECH}}(S^{3},\lambda)$ is equal to the infimum of the actions of periodic orbit in $\mathcal{S}_{1}(S^3,\lambda)$.
In \cite{Shi}, it is shown that  if $(L(2,1),\lambda)$ is strictly convex (or non-degenerate dynamically convex), then  the half of the first ECH spectrum $\frac{1}{2}c_{1}^{\mathrm{ECH}}(L(2,1),\lambda)$ is equal to the infimum of the actions of periodic orbit $\gamma$ in $\mathcal{S}_{2}(L(2,1),\lambda)$ with $\mu_{\tau_{\mathrm{glob}}}(\gamma)=1$. 
Here we note that $(L(p,q),\lambda)$ is called strictly convex if $(L(p,q),\lambda)$ is obtained as a quotient space of $(\partial S, \lambda_{0}|_{\partial S})$ of the action $(z_{1},z_{2})\mapsto (e^{\frac{2\pi i}{p}}z_{1},e^{\frac{2\pi iq}{p}}z_{2})$ where  $0\in S \subset \RR^{4}=\CC^{2}$ is a strictly convex domain which is invariant under the action.

The following is our second main theorem.

\begin{them}\label{maintheorem}
    Let $(L(3,1)\lambda)$ be a strictly convex (or non-degenerate dynamically convex) contact 3-manifold. Then
    \begin{equation*}
        \frac{1}{3}c_{1}^{\mathrm{ECH}}(L(3,1)\lambda)=\inf_{\gamma\in \mathcal{S}_{3},\,\,\mu_{\disk}(\gamma^3)=3}\,\,\int_{\gamma}\lambda.
    \end{equation*}
\end{them}

Note that our first result Theorem \ref{thm:main} plays an important role in the proof of  the above.

\subsection{Idea and outline of this paper}

At first, we prove Therorem \ref{thm:main}. The idea is as follows. Let $\gamma$ be a unkonotted orbit in a lens space. Take a tubular neighborhood, then we have a Heegaard splitting of genus 1. Consider the twist of the gluing map and compare the trivializations of the contact structure over the solid torus and a binding disk. 
By combining them with the properties of Conley-Zehnder index, we have Therorem \ref{thm:main}. The observation of the proof is also  essential in the following proof of  Theorem \ref{maintheorem}.

Next, we consider Theorem \ref{maintheorem}.

The estimate 
$ \inf_{\gamma\in \mathcal{S}_{3},\,\,\mu_{\mathrm{disk}}(\gamma^3)=3}\,\,\int_{\gamma}\lambda \leq  \frac{1}{3}c_{1}^{\mathrm{ECH}}(L(3,1)\lambda)$ is a hard part of the proof. At first, we conduct technical computations regarding indices present in ECH to clear which  holomorphic curves  appear as $U$-map to the empty set. As a result, it follows that any moduli space of holomorphic curves counted by $U$-map gives a structure of rational open book decomposition (Lemma \ref{lem:indechcomp} and Lemma \ref{lem:openbook}). Next, I consider the rational open book decompositions. 
By considering which actually supports $(L(3,1),\xi_{\mathrm{std}})$, we can narrow down the list of holomorphic curves. Then it turn out that any binding of the rational open book decomposition coming from the $U$-map $\gamma$ is in $\mathcal{S}_{3}$ and $\mu_{\mathrm{disk}}(\gamma^{3})=3$ and in addition we can choose it so that $\int_{\gamma}\lambda \leq \frac{1}{3} c_{1}^{\mathrm{ECH}}(L(3,1),\lambda)$ (Proposition \ref{prp;main}). In this manner, we have the theorem.

The estimate $\frac{1}{3}c_{1}^{\mathrm{ECH}}(L(3,1)\lambda)\leq \inf_{\gamma\in \mathcal{S}_{3},\,\,\mu_{\mathrm{disk}}(\gamma^3)=3}\,\,\int_{\gamma}\lambda$ is almost the same with \cite{Shi} except that we need to use  Theorem \ref{thm:main} and Theorem \ref{fundament} (i.e. a periodic orbit $\gamma$ in $\mathcal{S}_{3}$   must be elliptic if  $\mu_{\mathrm{disk}}(\gamma^{p})=3$).

\subsection*{Acknowledgement}
The author would like to thank his advisor Professor Kaoru Ono for his encouragement and support, Umberto Hryniewicz for conversation via e-mail and  A. Schneider for some comments. The author also would like to thank Takahiro Oba for sharing his knowledge of  contact topology.
This work was supported by JSPS KAKENHI Grant Number JP21J20300.

\section{Preliminaries}
\subsection{Conley-Zehnder index}
In this subsection, we recall Conley-Zehnder index and its properties introduced in \cite{HWZ2}. The contents in this subsection are based on \cite{HWZ2,HWZ4} and almost the same with what are summarized in \cite{Shi}.

\begin{dfn}
 For  a  smooth path $\varphi:\mathbb{R}\to Sp(1)$ in 2-dimensional symplectic matrices with $\varphi(0)=id$ and $\varphi(t+1)=\varphi(t)\varphi(1)$ for any $t\in \mathbb{R}$, the Conley-Zehnder index  $\mu_{CZ}(\varphi)\in \ZZ$ is defined. In particular, $ \mu_{CZ}$ is lower semi-continuous.
\end{dfn}
For $k\in \mathbb{Z}_{>0}$, define $\rho_{k}:\mathbb{R} \to \mathbb{R}$ as $t \mapsto kt$. The next proposition holds.
\begin{prp}\label{conleycovering}
    Consider a smooth path $\varphi:\mathbb{R}\to Sp(1)$ in  symplectic matrices with $\varphi(0)=id$ and $\varphi(t+1)=\varphi(t)\varphi(1)$ for any $t\in \mathbb{R}$.
    \item[(1).] If $\mu_{CZ}(\varphi)=2n$ for $n\in \mathbb{Z}$, then $\mu_{CZ}(\varphi\circ \rho_{k})=2kn$ for every $k\in \mathbb{Z}_{>0}$.
       \item[(2).] If $\mu_{CZ}(\varphi)\geq3$, then $\mu_{CZ}(\varphi\circ \rho_{k})\geq 2k+1$ for every $k\in \mathbb{Z}_{>0}$.
\end{prp}

Consider a periodic orbit $\gamma:\mathbb{R}/T_{\gamma}\mathbb{Z}\to Y$ of $(Y,\lambda)$ and  a symplectic trivialization $\tau:\gamma^{*}\xi \to \mathbb{R}/T_{\gamma}\mathbb{Z} \times \mathbb{C}$. Then we have a symplectic path $ \mathbb{R}\ni t\mapsto \phi_{\gamma,\tau}(t):=\tau(\gamma(T_{\gamma}t)) \circ d\phi^{tT_{\gamma}}|_{\xi}\circ \tau^{-1}(\gamma(0))$ which satisfies $\phi_{\gamma,\tau}(t+1)=\phi_{\gamma,\tau}(t)\phi_{\gamma,\tau}(1) $ for any $t\in \mathbb{R}$.

Now, we define the Conley-Zender index of $\gamma$ with respect to a trivialization $\tau$ as
\begin{equation}
    \mu_{\tau}(\gamma):=\mu_{CZ}(\phi_{\gamma,\tau}).
\end{equation}
Note that  $ \mu_{\tau}$ is independent of the choice of a trivialization in the same homotopy class of $\tau$.
 Let $\mathcal{P}(\gamma)$ denote  the set of homotopy classes of symplectic trivializations  $\gamma^{*}\xi\to \RR/T_{\gamma}\ZZ \times \RR^2$.  Note that  a  symplectic trivialization $\tau \in \mathcal{P}(\gamma)$ induce naturally a symplectic trivialization on $(\gamma^{p})^{*}\xi$ and we use the same notation $\tau\in \mathcal{P}(\gamma^{p})$ for the induced  trivialization on $(\gamma^{p})^{*}\xi$ if there is no confusion.

For each $\tau \in \mathcal{P}(\gamma)$, take  sections $Z_{\tau},W_{\tau}:\mathbb{R}/T_{\gamma}\mathbb{Z}\to \gamma^{*}\xi$ so that the map $\gamma^{*}\xi \ni aZ_{\tau}+bW_{\tau}\mapsto a+ib\in \mathbb{C}$  gives a  symplectic trivialization of homotopy class $\tau$.  For $\tau,\tau'\in \mathcal{P}(\gamma)$, we define $\mathrm{wind}(\tau,\tau')\in \mathbb{Z}$ as follows.
Let $a_{\tau,\tau'},b_{\tau,\tau'}:\mathbb{R}/T_{\gamma}\mathbb{Z}\to \mathbb{R}$ be continuous functions  such that $Z_{\tau}(t)=a_{\tau,\tau'}(t)Z_{\tau'}(t)+b_{\tau,\tau'}(t)W_{\tau'}(t)$ for $t\in \mathbb{R}/T_{\gamma}\mathbb{Z}$. Let $\theta:[0,T_{\gamma}]\to \mathbb{R}$ be a continuous function so that $a_{\tau,\tau'}(t)+ib_{\tau,\tau'}(t)\in \mathbb{R}_{+}e^{i\theta(t)}$ for $t\in [0,T_{\gamma}]$. Then we define 
\begin{equation}
    \mathrm{wind}(\tau,\tau'):=\frac{\theta(T_{\gamma})-\theta(0)}{2\pi}\in \mathbb{Z}.
\end{equation}
It is obvious that $\mathrm{wind}(\tau,\tau')$ is independent of the choices $Z_{\tau^{(')}},W_{\tau^{(')}}$.
The following property is well-known and important. 
\begin{prp}\label{conleywind}
      Let $\gamma$ be a  periodic orbit in $(Y,\lambda)$. For $\tau\in \mathcal{P}(\gamma)$. Then for any $\tau,\tau' \in \mathcal{P}(\gamma)$, $\mu_{\tau}(\gamma)+2\mathrm{wind}(\tau,\tau')=\mu_{\tau'}(\gamma)$.
\end{prp}
If $\gamma^{n}$ is non-degenerate for every $n \in \mathbb{Z}_{>0}$, it is also well-known that the Conley-Zehnder index behave as follows. 
\begin{prp}\label{conleybasic}
    Let $\gamma$ be a orbit such that $\gamma^{n}$ is non-degenerate for every $n\in \mathbb{Z}_{>0}$.  Fix a trivialization $\tau$ of the contact plane over $\gamma$. Consider the Conley-Zehnder indices of the multiple covers with respect to $\tau$. Write $
    \mu_{\tau}(\gamma^{n}):=\mu_{CZ}(\phi_{\gamma,\tau}\circ \rho_{n})$.
    \item[(1).] If $\gamma$ is hyperbolic, $\mu_{\tau}(\gamma^{n})=n\mu_{\tau}(\gamma)$ for every $n\in \mathbb{Z}_{>0}$.
    \item[(2).] If $\gamma$ is elliptic, there is $\theta \in \mathbb{R}\backslash \mathbb{Q}$ such that $\mu_{\tau}(\gamma^{n})=2 \lfloor n\theta \rfloor +1$ for every $n \in \mathbb{Z}_{>0}$.

    We call $\theta$ the monodromy angle of $\gamma$.
\end{prp}

For more properties of the Conley-Zehnder index, see \cite{HWZ2,HWZ4}.

\subsection{The construction and properties of ECH}\label{echdef}
Here, we list the basic construction and properties of ECH. The content of this subsection is based on \cite{H1}, \cite{H2}, \cite{H3} almost the same with what are summarized in \cite{Shi}.
 
Let $(Y,\lambda)$ be a non-degenerate contact three manifold. For $\Gamma \in H_{1}(Y;\mathbb{Z})$, Embedded contact homology $\mathrm{ECH}(Y,\lambda,\Gamma)$ is defined. At first, we define the chain complex $(\mathrm{ECC}(Y,\lambda,\Gamma),\partial)$. In this paper, we consider ECH over $\mathbb{Z}/2\mathbb{Z}=\mathbb{F}$.

\begin{dfn} [{\cite[Definition 1.1]{H1}}]\label{qdef}
An orbit set $\alpha=\{(\alpha_{i},m_{i})\}$ is a finite pair of distinct simple periodic orbit $\alpha_{i}$ with positive integer $m_{i}$. $\alpha=\{(\alpha_{i},m_{i})\}$ is called an ECH generator If $m_{i}=1$ whenever $\alpha_{i}$ is hyperbolic orbit.
\end{dfn}
Define $[\alpha]=\sum m_{i}[\alpha_{i}] \in H_{1}(Y)$. For two orbit sets $\alpha=\{(\alpha_{i},m_{i})\}$ and $\beta=\{(\beta_{j},n_{j})\}$ with $[\alpha]=[\beta]$, we define  $H_{2}(Y,\alpha,\beta)$ to be the set of relative homology classes of 2-chains $Z$ in $Y$ with $\partial Z =\sum_{i}m_{i}\alpha_{i}-\sum_{j}m_{j}\beta_{j}$. This is an affine space over $H_{2}(Y)$. 

From now on, we fix a trivialization $\tau_{\gamma} \in \mathcal{P}(\gamma)$ of the contact surface $\xi$ on each simple orbit $\gamma$.  Let $\tau:=\{ \tau_{\gamma} \}_{\gamma}$.
\begin{dfn}[{\cite[Definition 1.5]{H1}}]
For $Z\in H_{2}(Y,\alpha,\beta)$, we define
\begin{equation}
    I(\alpha,\beta,Z):=c_{1}(\xi|_{Z},\tau)+Q_{\tau}(Z)+\sum_{i}\sum_{k=1}^{m_{i}}\mu_{\tau}(\alpha_{i}^{k})-\sum_{j}\sum_{k=1}^{n_{j}}\mu_{\tau}(\beta_{j}^{k}).
\end{equation}
We call $I(\alpha,\beta,Z)$ an ECH index. Here,  $\mu_{\tau}$ is the Conely Zhender index with respect to $\tau$ and $c_{1}(\xi|_{Z},\tau)$ is a reative Chern number  and $Q_{\tau}(Z)=Q_{\tau}(Z,Z)$. Moreover this is independent of $\tau$ (see  \cite{H1} for more details).
\end{dfn}

For $\Gamma \in H_{1}(Y)$, we define 
\begin{equation*}
    \mathrm{ECC}(Y,\lambda,\Gamma):= \bigoplus_{\alpha:\mathrm{ECH\,\,generator\,\,with\,\,}{[\alpha]=\Gamma}}\mathbb{F}\cdot \alpha.
\end{equation*}
The right hand side is a freely generated module over
$\mathbb{F}$ by ECH generators $\alpha$ such that $[\alpha]=\Gamma$.
To define the differential $\partial:\mathrm{ECC}(Y,\lambda,\Gamma)\to \mathrm{ECC}(Y,\lambda,\Gamma) $, we fix a generic  almost complex structure $J$  on $\mathbb{R}\times Y$ which satisfies $\mathbb{R}$-invariant, $J(\frac{d}{ds})=X_{\lambda}$, $J\xi=\xi$ and $d\lambda(\cdot,J\cdot)>0$. We call such a almost complex structure $J$ admissible.

We consider $J$-holomorphic curves  $u:(\Sigma,j)\to (\mathbb{R}\times Y,J)$ where
the domain $(\Sigma, j)$ is a punctured compact Riemann surface. Here the domain $\Sigma$ is
not necessarily connected.  Let $\gamma$ be a (not necessarily simple) Reeb orbit.  If a puncture
of $u$ is asymptotic to $\mathbb{R}\times \gamma$ as $s\to \infty$, we call it a positive end of $u$ at $\gamma$ and if a puncture of $u$ is asymptotic to $\mathbb{R}\times \gamma$ as $s\to -\infty$, we call it a negative end of $u$ at $\gamma$ ( see \cite{H1} for more details ).

 Let $\alpha=\{(\alpha_{i},m_{i})\}$ and $\beta=\{(\beta_{i},n_{i})\}$ be orbit sets. Let $\mathcal{M}^{J}(\alpha,\beta)$ denote the set of  $J$-holomorphic curves with positive ends
at covers of $\alpha_{i}$ with total covering multiplicity $m_{i}$, negative ends at covers of $\beta_{j}$
with total covering multiplicity $n_{j}$, and no other punctures. Moreover, in $\mathcal{M}^{J}(\alpha,\beta)$, we consider two
$J$-holomorphic curves   to be equivalent if they represent the same current in $\mathbb{R}\times Y$. We sometimes consider an element in $\mathcal{M}^{J}(\alpha,\beta)$ as the image in $\mathbb{R}\times Y$. For $u \in \mathcal{M}^{J}(\alpha,\beta)$, we naturally have $[u]\in H_{2}(Y;\alpha,\beta)$ and set $I(u)=I(\alpha,\beta,[u])$. Moreover we define
\begin{equation}
     \mathcal{M}_{k}^{J}(\alpha,\beta):=\{\,u\in  \mathcal{M}^{J}(\alpha,\beta)\,|\,I(u)=k\,\,\}
\end{equation}

Under this notations, we define $\partial_{J}:\mathrm{ECC}(Y,\lambda,\Gamma)\to \mathrm{ECC}(Y,\lambda,\Gamma)$ as 
\begin{equation}
    \partial_{J} \alpha =\sum_{\beta:\mathrm{\,\,ECH\,\,generator\,\,with\,\,}[\beta]=\Gamma} \# (\mathcal{M}_{1}^{J}(\alpha,\beta)/\mathbb{R})\cdot \beta.
\end{equation}
Note that the above counting is well-defined and $\partial_{J} \circ \partial_{J}=0$ (see \cite{H1,HT1,HT2}, Proposition \ref{indexproperties} ). Moreover, the homology defined by $\partial_{J}$ does not depend on $J$ and if $\mathrm{Ker}\lambda=\mathrm{Ker}\lambda'$ for non-degenerate $\lambda,\lambda'$, there is a natural isomorphism between $\mathrm{ECH}(Y,\lambda,\Gamma)$ and $\mathrm{ECH}(Y,\lambda',\Gamma)$.  Indeed, It is proved in \cite{T1} that there is a natural isomorphism between  ECH and a version of Monopole Floer homology defined in \cite{KM}.

Next, we recall (Fredholm) index.
For $u\in \mathcal{M}^{J}(\alpha,\beta)$, the  its (Fredholm) index is defined by
\begin{equation}
    \mathrm{ind}(u):=-\chi(u)+2c_{1}(\xi|_{[u]},\tau)+\sum_{k}\mu_{\tau}(\gamma_{k}^{+})-\sum_{l}\mu_{\tau}(\gamma_{l}^{-}).
\end{equation}
Here $\{\gamma_{k}^{+}\}$ is the set consisting of (not necessarilly simple) all positive ends of $u$ and $\{\gamma_{l}^{-}\}$ is that one of all negative ends.  Note that for generic $J$, if $u$ is connected and somewhere injective, then the moduli space of $J$-holomorphic
curves near $u$ is a manifold of dimension $\mathrm{ind}(u)$ (see \cite[Definition 1.3]{HT1}).

\subsubsection{$U$-map}
Let $Y$ be connected.
Then there is degree$-2$ map $U$.
\begin{equation*}\label{Umap}
    U:\mathrm{ECH}(Y,\lambda,\Gamma) \to \mathrm{ECH}(Y,\lambda,\Gamma).
\end{equation*}

To define this, choose a base point $z\in Y$ which is especially  not on the image of any Reeb orbits and let $J$ be generic.
Then define a map 
\begin{equation*}
     U_{J,z}:\mathrm{ECC}(Y,\lambda,\Gamma) \to \mathrm{ECC}(Y,\lambda,\Gamma).
\end{equation*}
by
\begin{equation*}
    U_{J,z} \alpha =\sum_{\beta:\mathrm{\,\,ECH\,\,generator\,\,with\,\,}[\beta]=\Gamma} \# \{\,u\in \mathcal{M}_{2}^{J}(\alpha,\beta)/\mathbb{R})\,|\,(0,z)\in u\,\}\cdot  \beta.
\end{equation*}
The above map $U_{J,z}$ commute with $\partial_{J}$ and we can define the $U$ map
as the induced map on homology.  Moreover, this map is independent on $z$ (for a generic $J$). See \cite[{\S}2.5]{HT3} for more details. Moreover, in the same reason as $\partial$, $U_{J,z}$ does not depend of $J$ (see \cite{T1}).

\subsubsection{Partition conditions of elliptic orbits}\label{subsecparti}
For $\theta\in \mathbb{R}\backslash \mathbb{Q}$, we define $S_{\theta}$ to be the set of positive integers $q$ such that $\frac{\lceil q\theta \rceil}{q}< \frac{\lceil q'\theta \rceil}{q'}$ for all $q'\in \{1,\,\,2,...,\,\,q-1\}$ and write $S_{\theta}=\{q_{0}=1,\,\,q_{1},\,\,q_{2},\,\,q_{3},...\}$ in increasing order. Also $S_{-\theta}=\{p_{0}=1,\,\,p_{1},\,\,p_{2},\,\,p_{3},...\}$.

\begin{dfn}[{\cite[Definition 7.1]{HT1}}, or {\cite[{\S}4]{H1}}]\label{defpartition}
For non negative integer $M$, we inductively define the incoming partition $P_{\theta}^{\mathrm{in}}(M)$ as follows.

For $M=0$, $P_{\theta}^{\mathrm{in}}(0)=\emptyset$ and for $M>0$,
\begin{equation*}
    P_{\theta}^{\mathrm{in}}(M):=P_{\theta}^{\mathrm{in}}(M-a)\cup{(a)}
\end{equation*}

where $a:=\mathrm{max}(S_{\theta}\cap{\{1,\,\,2,...,\,\,M\}})$.
 Define outgoing partition
 \begin{equation*}
      P_{\theta}^{\mathrm{out}}(M):= P_{-\theta}^{\mathrm{in}}(M).
 \end{equation*}
 
The standard ordering convention for $P_{\theta}^{\mathrm{in}}(M)$ or $P_{\theta}^{\mathrm{out}}(M)$ is to list the entries
in``nonincreasing'' order.
\end{dfn}

 Let $\alpha=\{(\alpha_{i},m_{i})\}$ and $\beta=\{(\beta_{i},n_{i})\}$. For  $u\in \mathcal{M}^{J}(\alpha,\beta)$, it can be uniquely
written as $u=u_{0}\cup{u_{1}}$ where $u_{0}$ are unions of all components which maps to $\mathbb{R}$-invariant cylinders in $u$ and $u_{1}$ is the rest of $u$.

\begin{prp}[{\cite[Proposition 7.15]{HT1}}]\label{indexproperties}
Suppose that $J$ is generic and $u=u_{0}\cup{u_{1}}\in \mathcal{M}^{J}(\alpha,\beta)$. Then
    \item[(1).] $I(u)\geq 0$
    \item[(2).] If $I(u)=0$, then $u_{1}=\emptyset$
    \item[(3).] If $I(u)=1$, then  $\mathrm{ind}(u_{1})=1$. Moreover  $u_{1}$ is embedded and does not intersect $u_{0}$.
    \item[(4).] If $I(u)=2$ and $\alpha$ and $\beta$ are ECH generators,  then $\mathrm{ind}(u_{1})=2$ .  Moreover  $u_{1}$ is embedded and does not intersect $u_{0}$.
\end{prp}
\begin{prp}[{\cite[Proposition 7.14, 7.15]{HT1}}]\label{partitioncondition} 
Let $\alpha=\{(\alpha_{i},m_{i})\}$ and $\beta=\{(\beta_{j},n_{j})\}$ be ECH generators. Suppose that $I(u)=1$ or $2$ for $u=u_{0}\cup{u_{1}}\in \mathcal{M}^{J}(\alpha,\beta)$. Define $P_{\alpha_{i}}^{+}$ by the set consisting of the multiplicities of the positive ends of $u_{1}$ at covers of $\alpha_{i}$. In the same way,  define $P_{\beta_{j}}^{-}$ for the negative end.
Suppose that $\alpha_{i}$ in $\alpha$ (resp. $\beta_{j}$ in $\beta$) is elliptic orbit with the monodromy angle $\theta_{\alpha_{i}}$ (resp. $\theta_{\beta_{j}}$). Then under the standard ordering convention, $P_{\alpha_{i}}^{+}$ (resp. $P_{\beta_{j}}^{-}$) is an initial segment of $P_{\theta_{\alpha_{i}}}^{\mathrm{out}}(m_{i})$ (resp. $P_{\theta_{\beta_{j}}}^{\mathrm{in}}(n_{j})$).
\end{prp}
\begin{rem}
    There are partition conditions with respect to hyperbolic orbits, but omitted because we don't use it in this paper.
\end{rem}

\subsubsection{$J_{0}$ index and topological complexity of $J$-holomorphic curve}

In this subsection, we recall the $J_{0}$ index.

\begin{dfn}[{\cite[{\S}3.3]{HT3}}]
Let $\alpha=\{(\alpha_{i},m_{i})\}$ and $\beta=\{(\beta_{j},n_{j})\}$ be orbit sets with $[\alpha]=[\beta]$.
For $Z\in H_{2}(Y,\alpha,\beta)$, we define
\begin{equation}
    J_{0}(\alpha,\beta,Z):=-c_{1}(\xi|_{Z},\tau)+Q_{\tau}(Z)+\sum_{i}\sum_{k=1}^{m_{i}-1}\mu_{\tau}(\alpha_{i}^{k})-\sum_{j}\sum_{k=1}^{n_{j}-1}\mu_{\tau}(\beta_{j}^{k}).
\end{equation}
\end{dfn}

\begin{dfn}[]
Let $u=u_{0}\cup{u_{1}}\in \mathcal{M}^{J}(\alpha,\beta)$. Suppose that $u_{1}$ is somewhere injective. Let $n_{i}^{+}$
be the number of positive ends of $u_{1}$ which are asymptotic to $\alpha_{i}$, plus 1 if $u_{0}$ includes the trivial cylinder $\mathbb{R}\times \alpha_{i}$ with some multiplicity. Likewise, let $n_{j}^{-}$ be the number of negative ends of $u_{1}$ which are asymptotic to $\beta_{j}$, plus 1 if $u_{0}$ includes the trivial cylinder $\mathbb{R}\times \beta_{j}$ with some multiplicity. 
\end{dfn}
Write $J_{0}(u)=J_{0}(\alpha,\beta,[u])$.
\begin{prp}[{\cite[Lemma 3.5]{HT3}} {\cite[Proposition 5.8]{H3}}]
Let $\alpha=\{(\alpha_{i},m_{i})\}$ and $\beta=\{(\beta_{j},n_{j})\}$ be admissible orbit sets, and let $u=u_{0}\cup{u_{1}}\in \mathcal{M}^{J}(\alpha,\beta)$. Then
\begin{equation}
    -\chi(u_{1})+\sum_{i}(n_{i}^{+}-1)+\sum_{j}(n_{j}^{-}-1)\leq J_{0}(u)
\end{equation}
If $u$ is counted by the ECH differential or the $U$-map, then the above equality holds. 
\end{prp}

\subsubsection{ECH spectrum}
The notion of ECH spectrum is introduced in \cite{H2}. At first, we consider the filtered ECH.

The action of an orbit set $\alpha=\{(\alpha_{i},m_{i})\}$ is defined by 
\begin{equation*}
A(\alpha)=\sum m_{i}A(\alpha_{i})=\sum m_{i}\int_{\alpha_{i}}\lambda. 
\end{equation*}
For any $L>0$,  $\mathrm{ECC}^{L}(Y,\lambda,\Gamma)$ denotes the subspace of  $\mathrm{ECC}(Y,\lambda,\Gamma)$ which is generated by ECH generators whose actions are less than $L$. Then $(\mathrm{ECC}^{L}(Y,\lambda,\Gamma),\partial_{J})$ becomes a subcomplex and the homology group $\mathrm{ECH}^{L}(Y,\lambda,\Gamma)$ is defined. 

It follows from the construction that there exists a canonical homomorphism $i_{L}:\mathrm{ECH}^{L}(Y,\lambda,\Gamma) \to \mathrm{ECH}(Y,\lambda,\Gamma)$. In addition, for non-degenerate contact forms $\lambda,\lambda'$ with $\mathrm{Ker}\lambda=\mathrm{Ker}\lambda'=\xi$, there is a canonical isomorphism $\mathrm{ECH}(Y,\lambda,\Gamma)\to \mathrm{ECH}(Y,\lambda',\Gamma)$ defined by  the cobordism maps for product cobordisms (see \cite{H2}). 
Therefore we can consider a pair of a group $\mathrm{ECH}(Y,\xi,\Gamma)$ and maps $j_{\lambda}:\mathrm{ECH}(Y,\lambda,\Gamma) \to \mathrm{ECH}(Y,\xi,\Gamma)$ for  any non-degenerate contact form $\lambda$ with $\mathrm{Ker}\lambda=\xi$ such that $\{j_{\lambda}\}_{\lambda}$ is compatible with the canonical map $\mathrm{ECH}(Y,\lambda,\Gamma)\to \mathrm{ECH}(Y,\lambda',\Gamma)$.
\begin{dfn}[{\cite[Definition 4.1, cf. Definition 3.4]{H2}}]
Let $Y$ be a closed oriented three manifold with a non-degenerate contact form $\lambda$ with $\mathrm{Ker}\lambda=\xi$ and $\Gamma \in H_{1}(Y,\mathbb{Z})$. If $0\neq \sigma \in \mathrm{ECH}(Y,\xi,\Gamma)$, define
\begin{equation*}\label{spect}
    c_{\sigma}^{\mathrm{ECH}}(Y,\lambda)=\inf\{L>0 |\, \sigma \in \mathrm{Im}(j_{\lambda}\circ i_{L}:\mathrm{ECH}^{L}(Y,\lambda,\Gamma) \to \mathrm{ECH}(Y,\xi,\Gamma))\, \}
\end{equation*}
If $\lambda$ is degenerate, define
\begin{equation*}
    c_{\sigma}^{\mathrm{ECH}}(Y,\lambda)=\sup\{c_{\sigma}(Y,f_{-}\lambda)\}=\inf\{c_{\sigma}(Y,f_{+}\lambda)\}
\end{equation*}
where the supremum is over functions $f_{-}:Y\to (0,1]$ such that $f\lambda$ is non-degenerate and the infimum is over smooth functions $f_{+}:Y\to [1,\infty)$ such that $f_{-}\lambda$ is non-degenerate. Note that $c_{\sigma}^{\mathrm{ECH}}(Y,\lambda)<\infty$ and this definition makes sense. See {\cite[Definition 4.1, Definition 3.4, \S 2.3]{H2}} for more details.
\end{dfn}

\begin{dfn}\cite[Subsection 2.2]{H2}
Let $Y$ be a closed oriented three manifold with a non-degenerate contact form $\lambda$ with $\xi$. Then there is a canonical element called the ECH contact invariant,
\begin{equation}
    c(\xi):=\langle \emptyset \rangle \in \mathrm{ECH}(Y,\lambda,0).
\end{equation}
where $\langle \emptyset \rangle$ is the equivalent class containing $\emptyset$.
Note that $\partial_{J} \emptyset =0$ because of the maximal principle. In addition, $c(\xi)$  depends only on the contact structure $\xi$. 
\end{dfn}
\begin{dfn}\label{echspect}\cite[Definition 4.3]{H2}
    If $(Y,\lambda)$ is a closed connected contact three-manifold with $c(\xi)\neq 0$, and if $k$ is a nonnegative integer, define
    \begin{equation}
        c_{k}^{\mathrm{ECH}}(Y,\lambda):=\min\{c_{\sigma}^{\mathrm{ECH}}(Y,\lambda)|\,\sigma\in \mathrm{ECH}(Y,\xi,0),\,\,U^{k}\sigma=c(\xi)\}
    \end{equation}
The sequence $\{c_{k}^{\mathrm{ECH}}(Y,\lambda)\}$ is called the  ECH spectrum of $(Y,\lambda)$.
\end{dfn}
\begin{prp}\label{properties1}
    Let $(Y,\lambda)$ be a closed connected contact three manifold.
    \item[(1).]   \begin{equation*}
         0=c_{0}^{\mathrm{ECH}}(Y,\lambda)<c_{1}^{\mathrm{ECH}}(Y,\lambda) \leq c_{2}^{\mathrm{ECH}}(Y,\lambda)...\leq \infty.
     \end{equation*}
    \item[(2).] For any  $a>0$ and positive integer $k$, 
    \begin{equation*}
         c_{k}^{\mathrm{ECH}}(Y,a\lambda)= ac_{k}^{\mathrm{ECH}}(Y,\lambda).
    \end{equation*}
    \item[(3).] Let $f_{1},f_{2}:Y\to (0,\infty)$ be smooth functions with $f_{1}(x)\leq f_{2}(x)$ for every $x\in Y$. Then \begin{equation*}
        c_{k}^{\mathrm{ECH}}(Y,f_{1}\lambda)\leq c_{k}^{\mathrm{ECH}}(Y,f_{2}\lambda).
    \end{equation*}
    \item[(4).] Suppose $c_{k}^{\mathrm{ECH}}(Y,f\lambda)<\infty$. The map 
    \begin{equation*}
        C^{\infty}(Y,\mathbb{R}_{>0})\ni f \mapsto c_{k}^{\mathrm{ECH}}(Y,f\lambda)\in \mathbb{R}
    \end{equation*}
     is continuous in $C^{0}$-topology on $C^{\infty}(Y,\mathbb{R}_{>0})$.
\end{prp}
\begin{proof}[\bf Proof of Proposition \ref{properties1}]
They follow from the properties of ECH. See \cite{H2}.
\end{proof}
\subsubsection{ECH on Lens spaces}
Now, we focus on lens spaces.

Since $H_{2}(L(p,q))=0$, we write ECH index of $\alpha$, $\beta$ as $I(\alpha,\beta)$ instead of $I(\alpha,\beta,Z)$ where $\{Z\}= H_{2}(L(p,q);\alpha,\beta)$.

Let $(L(p,q),\lambda)$ be a non-degenerate contact lens space and consider ECH of  $0\in H_{1}(L(p,q))$. Since $[\emptyset]=0$, there is  an absolute $\mathbb{Z}$-grading on $\mathrm{ECH}(L(p,q),\lambda,0)=\bigoplus_{k\in\mathbb{Z}}\mathrm{ECH}_{k}(L(p,q),\lambda,0)$ defined by ECH index relative to $\emptyset$ where $\mathrm{ECH}_{k}(L(p,q),\lambda,0)$ is as follows.
Define
\begin{equation*}
    \mathrm{ECC}_{k}(L(p,q),\lambda,0):= \bigoplus_{\alpha:\mathrm{ECH\,\,generator},\,{[\alpha]=0},\,I(\alpha,\emptyset)=k}\mathbb{F}\cdot \alpha.
\end{equation*}
Since  $\partial_{J} $  maps $\mathrm{ECC}_{*}(L(p,q),\lambda,0)$ to  $\mathrm{ECC}_{*-1}(L(p,q),\lambda,0)$ and $U_{J,z}$ does  $\mathrm{ECC}_{*}(L(p,q),\lambda,0)$ to $\mathrm{ECC}_{*-2}(L(p,q),\lambda,0)$, we have $\mathrm{ECH}_{k}(L(p,q),\lambda,0)$ and 
\begin{equation*}
    U:\mathrm{ECH}_{*}(L(p,q),\lambda,0) \to\mathrm{ECH}_{*-2}(L(p,q),\lambda,0).
\end{equation*}
Based on these understandings, the next follows.
\begin{prp}\label{lenssp}
 Let $(L(p,q),\lambda)$ be non-degenerate.
 \item[(1).]   If $k$ is even and non-negative,
    \begin{equation*}
        \mathrm{ECH}_{k}(L(p,q),\lambda,0)\cong \mathbb{F}.
    \end{equation*}
    If $k$ is odd or negative, $\mathrm{ECH}_{k}(L(p,q),\lambda,0)$ is zero. 
    Moreover, for $n\geq1$ the $U$-map
    \begin{equation*}
        U:\mathrm{ECH}_{2n}(L(p,q),\lambda,0)\to\mathrm{ECH}_{2(n-1)}(L(p,q),\lambda,0)
    \end{equation*}
    is isomorphism.
    \item[(2).] If $\mathrm{Ker}\lambda=\xi_{\mathrm{std}}$, $0\neq c(\xi_{\mathrm{std}})\in \mathrm{ECH}_{0}(L(p,q),\lambda,0)$. Therefore, we can define the ECH spectrum (Definition \ref{echspect}).
\end{prp}
\begin{proof}[\bf Proof of Proposition \ref{lenssp}]
See \cite{H2}.
\end{proof}
\section{Proof of Theorem \ref{thm:main}}

In this section, we prove Theorem \ref{thm:main}.

First things first, we recall the existence of the following local coordinate called Martinet tube which is useful for our arguments in this paper.

\begin{prp}\cite{HWZ1}
Let $(Y^3,\lambda)$ be a contact three manifold wiht $\mathrm{Ker}\lambda=\xi$.  For a simple orbit $\gamma$,    there is  a  diffeomorphism called Martinet tube $F:\RR/\ZZ \times \DD_{\delta} \to \Bar{U}$ for a sufficiently small $\delta>0$ such that $F(t,0)=\gamma(t)$ and 
 there exists a smooth function $f:\RR/\ZZ \times \DD \to (0,+\infty)$ satisfying $f(\theta,0)=T_{\gamma}$, $df(\theta,0)=0$ and $F^{*}\lambda=f(\theta,x+iy)(d\theta+xdy)$. Here $\DD_{\delta}$ is the disk with radius $\delta$.
 Note that $\mathrm{Ker}F^{*}\lambda|_{\RR/\ZZ\times \{0\}}=\mathrm{span}(\partial_{x},\partial_{y})$.  Let  $\tau_{F}:\gamma^{*}\xi \to \RR/\ZZ \times \RR^{2}$ denote the induced trivialization by $F$ which maps $a\partial_{x}+b\partial_{y}$ to $(a,b)\in \RR^{2}$ on each fiber.
\end{prp}

Remark that we can take $F$ so that the trivialization $\tau_{F}$ realizes any given homotopy class of  trivializations over  $\gamma$.

 Consider  $V_{1}=S^{1}\times \DD$, $V_{2}=S^{1}\times \DD$ and a gluing map $g:\partial V_{1}=S^{1}\times \partial \DD \to S^{1}\times \partial \DD=\partial V_{2}$ which is described as
\begin{equation*}
\begin{pmatrix}
a & b \\
c & d \\
\end{pmatrix}
\end{equation*}
in standard longitude-meridian coordinates on the torus where $a,b,c,d \in \ZZ$, $b>0$ and $ad-bc=1$.  Then, there is an orientation-preserving diffeomorphism from the glued manifold $V_{1}\cup_{g}V_{2}$ to $L(p,q)$ if and only if $b=p$ and in addition either $d=-q\,\,\mathrm{mod}\,\,p$ or    $dq=-1\,\,\mathrm{mod}\,\,p$ (as remarked, the orientation of $L(p,q)$ is induced by the $4$-dimensional ball). Note that the boundary of a meridian disk of $V_{1}$ is glued by $g$ along a $(p,d)$-cable curve on $\partial V_{2}=S^{1}\times \partial \DD$. 

Let $\gamma\subset (L(p,q),\lambda)$ be a $p$-unknotted Reeb orbits and $u:\DD \to L(p,q)$ be a rational Seifert surface of $\gamma^{p}$.  
Take a  Martinet tube  $F:\RR/T_{\gamma}\ZZ \times \DD_{\delta} \to \Bar{U}$ for a sufficiently small $\delta>0$ onto a small  open neighbourhood $\gamma \subset 
 \Bar{U}$.
 
 As Remark \ref{heegaard}, $\Bar{U}$ is a solid torus such that $L(p,q)\backslash U$ is also a solid torus, which gives a Heegaard decomposition of genus $1$.  In addition, $u(\DD)\cap (L(p,q)\backslash U)$ is a meridian disk of $L(p,q)\backslash U$. Therefore, $F^{-1}(u(\DD)\cap\partial \Bar{U})$ is a $(p,r)$ cable such that either $-r=q\,\,\mathrm{mod}\,\,p$ or   $-rq=1\,\,\mathrm{mod}\,\,p$ with respect to the coordinate of $\RR/T_{\gamma}\ZZ \times \DD $. Therefore, Theorem \ref{thm:main} follows directly from the next proposition.

 \begin{prp}\label{main;prop}
     Suppose that the above $F^{-1}(u(\DD)\cap\partial \Bar{U})$ is a $(p,r)$ cable. If $-2r-2p\cdot sl^{\mathbb{Q}}_{\xi}(\gamma)-\mu_{\mathrm{disk}}(\gamma^{p})$  is not divisible by $p$, then  $\gamma$ is elliptic.
 \end{prp}

To prove the above proposition, we take sections $Z_{\mathrm{disk}}:\RR/pT_{\gamma}\ZZ \to (\gamma^p)^{*}\xi$ and $Z_{F}:\RR/pT_{\gamma}\ZZ \to (\gamma^p)^{*}\xi$ so that $Z_{\mathrm{disk}}$ extends to a non-vanishing section on $u^{*}\xi$ and $Z_{F}$  corresponds to $\partial_{x}$ on the coordinate induced by $F$. Let $Z^{\epsilon}_{\mathrm{disk}}:\RR/pT_{\gamma}\ZZ \to L(p,q)$ and $Z^{\epsilon}_{F}:\RR/pT_{\gamma}\ZZ \to L(p,q)$ denote the curves  $Z^{\epsilon}_{\mathrm{disk}}(t)=\mathrm{exp}_{\gamma(t)}(\epsilon Z_{\mathrm{disk}}(t))$ and $Z^{\epsilon}_{F}(t)=\mathrm{exp}_{\gamma(t)}(\epsilon Z_{F}(t))$ for small $\epsilon>0$ respectively.  
Then, it follows from a direct observation and the definition that $\# (u(\DD)\cap Z^{\epsilon}_{0})=-rp$ and  $\# (u(\DD)\cap Z^{\epsilon}_{\mathrm{disk}})=p^{2}sl_{\xi}^{\mathbb{Q}}(\gamma)$ (note the orientation and sign).
 
Let $\rho:S^{3}\to L(p,q)$ be the covering map. We can take lifts of $\gamma^{p}:\RR/pT_{\gamma}\ZZ \to L(p,q)$ and the rational Seifert surface  $u:\DD \to L(p,q)$ to $S^{3}$, and write $\Tilde{\gamma}: \RR/pT_{\gamma}\ZZ \to S^{3}$, $\Tilde{u}:\DD \to S^{3}$ respectively. We may assume that $\Tilde{\gamma}(pT_{\gamma}t)=\Tilde{u}(e^{2\pi t})$. 
In the same way, we take lifts of $Z_{\mathrm{disk}},\,Z_{F}:\RR/pT_{\gamma}\ZZ \to (\gamma^p)^{*}\xi$ and write $\Tilde{Z}_{\mathrm{disk}},\,\Tilde{Z}_{F}:\RR/pT_{\gamma}\ZZ \to \Tilde{\gamma}^{*}\xi$ respectively. Let $\Tilde{Z}^{\epsilon}_{\mathrm{disk}}(t)=\mathrm{exp}_{\Tilde{\gamma}(t)}(\epsilon \Tilde{Z}_{\mathrm{disk}}(t))$ and $\Tilde{Z}^{\epsilon}_{F}(t)=\mathrm{exp}_{\Tilde{\gamma}(t)}(\epsilon \Tilde{Z}_{F}(t))$. 
Then it follows from the construction that $\Tilde{Z}^{\epsilon}_{\mathrm{disk}}$ and $\Tilde{Z}^{\epsilon}_{F}$ are lifts of $Z^{\epsilon}_{\mathrm{disk}}$ and $Z^{\epsilon}_{F}$ respectively. 
Since there are $p$ ways of lifting $Z^{\epsilon}_{\mathrm{disk}}$, $Z^{\epsilon}_{F}$  and each intersection number of a lift with $\Tilde{u}(\DD)$ is equal to each other, we have
\begin{equation*}
\# (u(\DD)\cap Z^{\epsilon}_{F})=p\# (\Tilde{u}(\DD)\cap \Tilde{Z}^{\epsilon}_{F}),\,\,\# (u(\DD)\cap Z^{\epsilon}_{\mathrm{disk}})=p\# (\Tilde{u}(\DD)\cap \Tilde{Z}^{\epsilon}_{\mathrm{disk}})
\end{equation*}
and hence $\# (\Tilde{u}(\DD)\cap \Tilde{Z}^{\epsilon}_{F})=-r,\,\,\# (\Tilde{u}(\DD)\cap \Tilde{Z}^{\epsilon}_{\mathrm{disk}})=psl_{\xi}^{\mathbb{Q}}(\gamma)$.

   Recall that $\tau_{F}:\gamma^{*}\xi\to \RR/T_{\gamma}\times \RR^{2}$ denotes the  trivialization induced by $F$. As mentioned, we use the same notation $\tau_{F}$  for a trivialization $(\gamma^{p})^{*}\xi\to \RR/T_{\gamma}\times \RR^{2}$ induced by $\tau_{F}:\gamma^{*}\xi\to \RR/T_{\gamma}\times \RR^{2}$.
   
   \begin{lem}\label{lem;key}
       $\mu_{\tau_{F}}(\gamma^{p})-2psl_{\xi}^{\mathbb{Q}}(\gamma)-2r=\mu_{\mathrm{disk}}(\gamma^{p})$
   \end{lem}
\begin{proof}[\bf Proof of Lemma \ref{lem;key}]
Take a section $\Tilde{W}_{\mathrm{disk}}:\RR/pT_{\gamma}\ZZ \to (\Tilde{\gamma})^{*}\xi$ so that the map $\Tilde{\gamma}^{*}\xi \ni a\Tilde{Z}_{\mathrm{disk}}+b\Tilde{W}_{\mathrm{disk}}\mapsto a+ib\in \RR^{2}$  gives a  symplectic trivialization. 
Since  $\Tilde{Z}_{\mathrm{disk}}:\RR/pT_{\gamma}\ZZ \to (\Tilde{\gamma})^{*}\xi$ extends globally to a non-vanishing section on $\Tilde{u}^{*}\xi$, the homotopy class of this trivialization is  $\tau_{\mathrm{disk}}$.
In the same way, we  take a section $\Tilde{W}_{F}:\RR/pT_{\gamma}\ZZ \to (\Tilde{\gamma})^{*}\xi$ so that the map $\Tilde{\gamma}^{*}\xi \ni a\Tilde{Z}_{F}+b\Tilde{W}_{F}\mapsto a+ib\in \RR^{2}$  gives a  symplectic trivialization. Since $\Tilde{Z}_{F}:\RR/pT_{\gamma}\ZZ \to (\Tilde{\gamma})^{*}\xi$ corresponds to $\partial_{x}$ with respect to the coordinate induced by $F$, the homotopy class of this trivialization is  $\tau_{F}$. Therefore it follows directly that
\begin{equation}\label{wind}
   \mathrm{wind}(\tau_{F},\tau_{\mathrm{disk}})=\# (\Tilde{u}(\DD)\cap \Tilde{Z}^{\epsilon}_{F})-\# (\Tilde{u}(\DD)\cap \Tilde{Z}^{\epsilon}_{\mathrm{disk}})=-r-psl_{\xi}^{\mathbb{Q}}(\gamma).
\end{equation}
It follows from Proposition \ref{conleywind} that 
\begin{equation}
    \mu_{\tau_{F}}(\gamma^{p})+2\mathrm{wind}(\tau_{F},\tau_{\mathrm{disk}})=\mu_{\tau_{F}}(\gamma^{p})-2r-2psl_{\xi}^{\mathbb{Q}}(\gamma)=\mu_{\mathrm{disk}}(\gamma^{p}).
\end{equation}
This completes the proof.
\end{proof}
Now, we shall complete the proof of Proposition \ref{main;prop}. Suppose that $\gamma \subset L(p,q)$ is hyperbolic. Then according to Proposition \ref{conleybasic},  $\mu_{\tau_{F}}(\gamma^{p})=p\mu_{\tau_{F}}(\gamma)$. Therefore it follows from Lemma \ref{lem;key} that $-p\mu_{\tau_{F}}(\gamma)=-2psl_{\xi}^{\mathbb{Q}}(\gamma)-2r-\mu_{\mathrm{disk}}(\gamma^{p})$. This means that if $-2psl_{\xi}^{\mathbb{Q}}(\gamma)-2r-\mu_{\mathrm{disk}}(\gamma^{p})$ is not divisible by $p$, then $\gamma$ must be elliptic. This completes the proof.

\section{Immersed $J$-holomorphic curves}
In this sectin, we assume that $Y\cong L(3,1)$ and $(Y,\lambda)$ is non-degenerate dynamically convex.  From now, we fix a generic admissible almost complex structure $J$ (c.f. \S\ref{echdef}) and  trivialization $\tau_{\gamma} \in \mathcal{P}(\gamma)$ of the contact surface $\xi$ on each simple orbit $\gamma$.  Let $\tau:=\{ \tau_{\gamma} \}_{\gamma}$.
\begin{prp}\label{prp:immersed}
    Let $u:(\Sigma,j)\to (\RR\times Y,J)$ be an immersed $J$-holomorphic curve with no negative end.   
    \item[(1).] $\ind(u)$ is not equal to $1$.
    \item[(2).] If  $\ind(u)=2$, then $u$ is of genus $0$.
\end{prp}
\begin{proof}[\bf Proof of Proposition \ref{prp:immersed}]
 Let $g(u)$ denote the genus of $u$.  Since $u$ is immersion, we have
 \begin{equation}\label{eq:ind}
 \begin{split}
     \ind(u)=&-(2-2g(u)-k)+2c_{\tau}(\xi|_{[u]})+\sum_{1\leq i \leq k} \mu_{\tau}(\gamma_{i})\\
     =&-(2-2g(u))+2c_{\tau}(\xi|_{[u]})+\sum_{1\leq i \leq k} (\mu_{\tau}(\gamma_{i})+1)
      \end{split}
 \end{equation}
 where $\{\gamma_{i}\}_{i}$ is the set of periodic orbits to which the ends of $u$ are asymptotic. We may assume that $\gamma_{i}$ is hyperbolic if $1\leq i\leq k'$ and elliptic if $k'+1\leq i \leq k$.
 
 \begin{lem}\label{lem:iterate}
     Let a periodic orbit $\gamma$ be elliptic. For any symplectic trivialization $\tau:\gamma^{*}\xi \to \gamma \times \RR^{2}$ and $p\in \ZZ_{>0}$, we have
\begin{equation} 
         p\mu_{\tau}(\gamma)-\mu_{\tau}(\gamma^{p})+p\geq 1.
 \end{equation}
 \end{lem}
 \begin{proof}[\bf Proof of Lemma \ref{lem:iterate}]
 Let $\theta$ denote the monodolomy angle with respect to $\tau$. Then $\mu_{\tau}(\gamma^{p})=2\lfloor p\theta \rfloor+1$ for any $p\in \ZZ_{>0}$. Since either $\lfloor2 \theta \rfloor=2\lfloor \theta \rfloor+1$ or $\lfloor2 \theta \rfloor=2\lfloor \theta \rfloor$, we have $2\lfloor \theta \rfloor +1 \geq \lfloor2 \theta \rfloor$.  It is easy to check by inducting that $p\lfloor \theta \rfloor +p-1 \geq \lfloor p\theta \rfloor$. This completes the proof.
 \end{proof}
Recall  $Y \cong L(3,1)$. Hence for any periodic orbit $\gamma$, $\gamma^{3}$ is contractible. Let $\alpha$ denote the orbit set to which the positive ends of $u$ are asymptotic. Then $\{[u]\} = H_{2}(Y,\alpha,\emptyset)$.   
Now, we note some obvious facts.  First,  $c_{\tau}(\xi|_{3[u]})=3c_{\tau}(\xi|_{[u]})$ where $c_{\tau}$ is the first relative Chern number and $\{3[u]\} = H_{2}(Y,3\alpha,\emptyset)$.  Next, since any $\gamma_{i}^{3}$ is contractible,  $2c_{\tau}(\xi|_{3[u]})+\sum_{1\leq i \leq k} \mu_{\tau}(\gamma_{i}^{3})=\sum_{1\leq i \leq k} \mu_{\disk}(\gamma_{i}^{3})$.

Based on this understanding, we multiply both sides of (\ref{eq:ind}) by $3$. Then  we have
 \begin{equation}\label{eq:ineq}
 \begin{split}
      3\ind(u)=&-3(2-2g(u))+2c_{\tau}(\xi|_{3[u]})+\sum_{1\leq i \leq k'} (\mu_{\tau}(\gamma_{i}^{3})+3)\\ &+\sum_{k'+1\leq i \leq k} \mu_{\tau}(\gamma_{i}^{3})+\sum_{k'+1\leq i \leq k}( p\mu_{\tau}(\gamma_{i})-\mu_{\tau}(\gamma_{i}^{3})+3)\\
      =&-3(2-2g(u))+\sum_{1\leq i \leq k'}( \mu_{\mathrm{disk}}(\gamma_{i}^{3})+3)\\+&\sum_{k'+1\leq i \leq k} \mu_{\mathrm{disk}}(\gamma_{i}^{3})+\sum_{k'+1\leq i \leq k}( 3\mu_{\tau}(\gamma_{i})-\mu_{\tau}(\gamma_{i}^{3})+3)\\
      \geq &6g(u)-6+\sum_{1\leq i \leq k'}( \mu_{\mathrm{disk}}(\gamma_{i}^{3})+3)+\sum_{k'+1\leq i \leq k}( \mu_{\mathrm{disk}}(\gamma_{i}^{3})+1)
      \end{split}
 \end{equation}
 Here Lemma \ref{lem:iterate} is used.

 Suppose that  $\ind(u)=1$. From (\ref{eq:ineq}), we have
 \begin{equation}
     9\geq 6g(u)+\sum_{1\leq i \leq k'}( \mu_{\mathrm{disk}}(\gamma_{i}^{3})+3)+\sum_{k'+1\leq i \leq k}( \mu_{\mathrm{disk}}(\gamma_{i}^{3})+1)
 \end{equation}
 We note that $\mu_{\mathrm{disk}}(\gamma_{i}^{3})+3$ is at least $6$ and $\sum_{k'+1\leq i \leq k}( \mu_{\mathrm{disk}}(\gamma_{i}^{3})+1)$ is at least $4$ because of dynamical convexity. Since $\mathrm{ind}(u)=1$ is odd, it follows from  (\ref{eq:ind}) that at least one $\gamma_{i}$ must be positive hyperbolic.
 Therefore $k'\geq 1$ and thus $k'=k=1$. This means that the only one orbit $\gamma_{1}$ is contractible. Now we go back to  (\ref{eq:ind}). We have $1=2g(u)-1+\mu_{\mathrm{disk}}(\gamma_{1})$. This does not happen since $\mu_{\mathrm{disk}}(\gamma_{1})\geq 3$ and $g(u)\geq 0$. This proves Proposition \ref{prp:immersed}(1).

 Next we suppose that $\ind(u)=2$ and $g(u)\geq 1$. From (\ref{eq:ineq}), we have
 \begin{equation}
     6\geq \sum_{1\leq i \leq k'}( \mu_{\mathrm{disk}}(\gamma_{i}^{3})+3)+\sum_{k'+1\leq i \leq k}( \mu_{\mathrm{disk}}(\gamma_{i}^{3})+1)
 \end{equation}
 In the same way as above, we have $k=1$ and so from (\ref{eq:ind}) $2=\ind(u)=2g(u)-1+\mu_{\mathrm{disk}}(\gamma_{1})$. But this does not happen since $\mu_{\mathrm{disk}}(\gamma_{1})\geq 3$ and $g(u)\geq 1$. This proves Proposition \ref{prp:immersed}(2).
\end{proof}
\begin{rem}
   In general, the above argument does not work for any $L(p,q)$.
\end{rem}

\begin{prp}\label{prp;main}
    Let $\alpha$ be an ECH generator with $\langle U_{J,z}\alpha,\emptyset \rangle\neq 0$. Then $\alpha$ satisfies one of the following.
    \item[(1).] There are simple elliptic orbits $\gamma_{1},\,\gamma_{2},\,\gamma_{3} \in  \mathcal{S}_{3}$ with $\mu_{\mathrm{disk}}(\gamma^{3}_{i})=3$ for $i=1,2,3$ such that $\alpha=(\gamma_{1},1)\cup(\gamma_{2},1)\cup(\gamma_{3},1)$.
    \item [(2).] There is a simple elliptic orbit $\gamma \in  \mathcal{S}_{3}$ with $\mu_{\disk}(\gamma^{3})=3$ such that $\alpha=(\gamma,3)$.
    \item [(3).]There are  simple elliptic orbits $\gamma_{1},\,\gamma_{2}\in  \mathcal{S}_{3}$ with $\mu_{\disk}(\gamma_{1}^{3})=\mu_{\disk}(\gamma_{2}^{3})=3$  such that $\alpha=(\gamma_{1},2)\cup(\gamma_{2},1)$.
\end{prp}
There are many steps to prove Proposition \ref{prp;main}. At first, we compute some indices and list the properties of  ECH generators $\alpha$ which satisfy $\langle U_{z,J}\alpha,\emptyset \rangle\neq 0$.
\begin{lem}\label{lem:indechcomp}
    Let $\alpha$ be an ECH generator with $\langle U_{J,z}\alpha,\emptyset \rangle\neq 0$. Then   by conducting some computations regarding indices, we have that any $u\in \mathcal{M}^{J}(\alpha,\emptyset)$ has no two ends asymptotic to the same orbit.  In addition, $\alpha$ satisfies one of the following.
    \item[(1).] There are simple elliptic orbits $\gamma_{1}$, $\gamma_{2}$, $\gamma_{3}$ with $\mu_{\mathrm{disk}}(\gamma^{3}_{i})=3$ for $i=1,2,3$ such that $\alpha=(\gamma_{1},1)\cup(\gamma_{2},1)\cup(\gamma_{3},1)$.
    \item [(2).] There is a simple elliptic orbit $\gamma$ with $\mu_{\disk}(\gamma^{3})=3$ such that $\alpha=(\gamma,3)$.
    \item [(3).]There are simple elliptic orbits $\gamma_{1}$ and $\gamma_{2}$  with $\mu_{\disk}(\gamma_{1}^{3})=\mu_{\disk}(\gamma_{2}^{3})=3$ such that $\alpha=(\gamma_{1},2)\cup(\gamma_{2},1)$.
    \item[(4).]There are  simple elliptic orbits $\gamma_{1}$ $\gamma_{2}$  with $\mu_{\disk}(\gamma_{1}^{6})=5$ and  $\mu_{\disk}(\gamma_{2}^{3})=5$  such that $\alpha=(\gamma_{1},2)\cup(\gamma_{2},1)$.
    \item [(5).]There are  simple  orbits $\gamma_{1}$ and  $\gamma_{2}$ such that $\alpha=(\gamma_{1},1)\cup(\gamma_{2},1)$ and each of them is not positive hyperbolic.
    \item [(6).] There are   simple elliptic orbits $\gamma_{1}$ and  $\gamma_{2}$ such that $\alpha=(\gamma_{1},2)\cup(\gamma_{2},2)$.
    \item[(7).] There is a simple orbit $\gamma$ such that $\alpha=(\gamma,1)$ and $\gamma$ is not positive hyperbolic. 
\end{lem}
\begin{proof}[\bf Proof of Lemma \ref{lem:indechcomp}]
Set $\alpha=\{(\gamma_{i},m_{i})\}_{1\leq i \leq k}$. We may assume that $\gamma_{i}$ is hyperbolic for $1\leq i \leq k'$ and elliptic for $k'+1\leq i \leq k$. Of course, $m_{i}=1$ for  $1\leq i \leq k'$ since $\alpha$ is an ECH generator. Let $u\in \mathcal{M}^{J}(\alpha,\emptyset)$ be a $J$-holomorphic curve counted by $\langle U_{z,J}\alpha,\emptyset \rangle \neq 0$. Then we have
\begin{equation}
\begin{split}
     2c_{\tau}(\xi|_{[u]})+\sum_{1\leq i \leq k} \mu_{\tau}(\gamma_{i}^{m_{i}})=&I(u)-J_{0}(u)\\=&2-(-\chi(u)+\sum_{1\leq i \leq k}(n_{i}^{+}-1))
     \\=&2-(2g(u)-2+h+\sum_{1\leq i \leq k}(n_{i}^{+}-1)).
\end{split}
\end{equation}
Here $h$ is the number of the positive ends of $u$ and $n_{i}^{+}$ is the number of the positive ends asymptotic to $\gamma_{i}$. It follows from Proposition \ref{prp:immersed} (2) that $g(u)=0$. Hence  we have
\begin{equation}
\begin{split}
    4=h+\sum_{1\leq i \leq k}(n_{i}^{+}-1)+2c_{\tau}(\xi|_{[u]})+\sum_{1\leq i \leq k} \mu_{\tau}(\gamma_{i}^{m_{i}})
    \end{split}
\end{equation}
Now, we conduct  similar calculations with the proof of Proposition \ref{prp:immersed}. Multiplying both side by $3$, it follows from Lemma \ref{lem:iterate} that
\begin{equation}\label{ineqcalc}
    \begin{split}
           12=&3h+3\sum_{1\leq i \leq k}(n_{i}^{+}-1)+2c_{\tau}(\xi|_{3[u]})+\sum_{1\leq i \leq k'} \mu_{\tau}(\gamma_{i}^{3})\\+&\sum_{k'+1\leq i \leq k} \mu_{\tau}(\gamma_{i}^{3m_{i}})+\sum_{k'+1\leq i \leq k}( 3\mu_{\tau}(\gamma_{i}^{m_{i}})-\mu_{\tau}(\gamma_{i}^{3m_{i}}))\\=&3h+3\sum_{1\leq i \leq k}(n_{i}^{+}-1)+\sum_{1\leq i \leq k'} (\mu_{\disk}(\gamma_{i}^{3})+3)\\+&\sum_{k'+1\leq i \leq k} \mu_{\disk}(\gamma_{i}^{3m_{i}})+\sum_{k'+1\leq i \leq k}( 3\mu_{\tau}(\gamma_{i}^{m_{i}})-\mu_{\tau}(\gamma_{i}^{3m_{i}})+3)-3k\\ \geq& 3(h-k)+3\sum_{1\leq i \leq k}(n_{i}^{+}-1)\\+&\sum_{1\leq i \leq k'} (\mu_{\disk}(\gamma_{i}^{3})+3)+\sum_{k'+1\leq i \leq k}( \mu_{\disk}(\gamma_{i}^{3m_{i}})+1).
    \end{split}
\end{equation}
\begin{cla}\label{cla;n=1}
    $n_{i}^{+}=1$ for any $1\leq i \leq k$. That is, the number of the ends of $u$ asymptotic to $\gamma_{i}$ is $1$ for any $1\leq i \leq k$. This means $h=k$ and $\sum_{1\leq i \leq k}(n_{i}^{+}-1)=0$ in (\ref{ineqcalc}).
\end{cla}
\begin{proof}[\bf Proof of Claim \ref{cla;n=1}]
    We prove this by contradiction. Assume that $n^{+}_{j}>1$ for some $k'+1\leq j\leq k$. Then $3(h-k)\geq 3$ and $3\sum_{1\leq i \leq k}(n_{i}^{+}-1)\geq 3$.  Therefore we  have
    \begin{equation}\label{equat}
 6\geq \sum_{1\leq i \leq k'}(\mu_{\disk}(\gamma_{i}^{3})+3)+\sum_{k'+1\leq i \leq k}( \mu_{\disk}(\gamma_{i}^{3m_{i}})+1).
  \end{equation}
  Moreover $m_{j}>1$. It follows easily from (\ref{equat}) that $k=1$. Indeed If $k>1$,  the right hand side of (\ref{equat}) is at least $8$. Hence $k=j=1$.
  If $\gamma_{1}$ is contractible, $\mu_{\disk}(\gamma^{3m_{1}})\geq 6m_{1}+1\geq 13$. This is a contradiction. So $\gamma_{1}$ must be non-contractible. 
  Suppose that $\gamma$ is non-contractible. Since $[\alpha]=m_{1}[\gamma_{1}]$ must be zero in $H_{1}(L(3,1))$,  $m_{1}$ is divisible by $3$. Write $m_{1}=3m'_{1}$. We have $\mu_{\disk}(\gamma^{3m_{1}})+1=\mu_{\disk}((\gamma^{3})^{3m'_{1}})+1\geq 6m'_{1}+1+1\geq 8$ (Proposition \ref{conleybasic}). This contradicts (\ref{equat}). In summary, we have  $n^{+}_{i}=1$ for any $i$. This completes the proof.
\end{proof}
Having Claim \ref{cla;n=1},  it follows from (\ref{ineqcalc}) that
\begin{equation}\label{ineqkk'}
\begin{split}
     12\geq& \sum_{1\leq i \leq k'}(\mu_{\disk}(\gamma_{i}^{3})+3)+\sum_{k'+1\leq i \leq k}( \mu_{\disk}(\gamma_{i}^{3m_{i}})+1)\\
     \geq& 6k'+4(k-k').
     \end{split}
\end{equation}
The pair $(k-k',k')$ must satisfy (\ref{ineqkk'}). Thus  $(k-k',k')$ is one of the following. $(k-k',k')=(3,0),(0,2),(1,1),(2,0),(0,1),(1,0)$. We check their properties one by one.

The next claim is obvious but it is worth to be mentioned explicitly for further arguments.

\begin{cla}\label{cla;essential}
    Suppose that $\mu_{\tau}(\gamma)\geq3 $. If $\mu_{\tau}(\gamma^{k})=5$ or $7$ for $k>1$, then $\mu_{\tau}(\gamma)=3$. In addition if $\mu_{\tau}(\gamma^{k})=5$ for $k>1$, then $k=2$ and  if $\mu_{\tau}(\gamma^{k})=7$ for $k>1$, then either $k=2$ or $k=3$.
\end{cla}
\begin{proof}[\bf Proof of Claim \ref{cla;essential}]
Note that  $\gamma$ is elliptic if $\mu_{\tau}(\gamma^{k})=5$ or $7$ for $k>1$. Indeed   if $\gamma$ is hyperbolic, then $\mu_{\tau}(\gamma^{k})=k\mu_{\tau}(\gamma)$ (c.f. Proposition \ref{conleybasic}), but since 5 and 7 are prime, this is a contradiction.  
Since $\gamma$ is elliptic, there is $\theta \in  \RR \backslash \QQ$ such that $\mu_{\tau}(\gamma^{k})=2\lfloor k\theta \rfloor+1$ for any $k\geq1$ (Proposition \ref{conleybasic}). 
If $\mu_{\tau}(\gamma)\geq 5$, then $\theta\geq 2$ because $\mu_{\tau}(\gamma)=2 \lfloor \theta \rfloor+1\geq 5$.
Therefore $\mu_{\tau}(\gamma^{k})=2\lfloor k \theta \rfloor+1\geq 4k+1$ (c.f. Proposition \ref{conleycovering}). But any $k>1$ does not satisfy $\mu_{\tau}(\gamma^{k})=5$ and $7$. This is a contradiction. Thus we have $\mu_{\tau}(\gamma)=3$. 

If $\mu_{\tau}(\gamma^{k})=5$ for $k>1$,  $k=2$ follows directly  from Proposition \ref{conleycovering}. In the same way, we have if $\mu_{\tau}(\gamma^{k})=7$ for $k>1$, then either $k=2$ or $k=3$.
\end{proof}

\item[(a) If $(k-k',k')=(3,0)$.]  

In this case, we have $\alpha=(\gamma_{1},m_{1})\cup (\gamma_{2},m_{2})\cup(\gamma_{3},m_{3})$ for some elliptic orbits $\gamma_{i}$. Moreover from (\ref{ineqkk'}), $\mu_{\disk}(\gamma_{i}^{3m_{i}})=3$ for $i=1,2,3$. This implies that $m_{i}=1$ (Proposition \ref{conleycovering}) and thus $\mu_{\disk}(\gamma_{i}^{3})=3$ for $i=1,2,3$. Thus $\alpha$ satisfies  (1) in Lemma \ref{lem:indechcomp}.

\item[(b) If $(k-k',k')=(0,2)$.]  In this case, we have $\alpha=(\gamma_{1},1)\cup (\gamma_{2},1)$ with $\mu_{\disk}(\gamma_{1}^3)=\mu_{\disk}(\gamma_{2}^3)=3$. Thus $\gamma_{1}$ and $\gamma_{2}$ are negative hyperbolic, and $\alpha$ satisfies  (4) in Lemma \ref{lem:indechcomp}.

\item[(c) If $(k-k',k')=(1,1)$.]
In this case,  $\alpha=(\gamma_{1},1)\cup (\gamma_{2},m_{2})$ for elliptic $\gamma_{2}$ and 
negative hyperbolic $\gamma_{1}$.  Since $ 12\geq (\mu_{\disk}(\gamma_{1}^{3})+3)+( \mu_{\disk}(\gamma_{2}^{3m_{2}})+1)$, it follows from dynamical convexity that either $(\mu_{\disk}(\gamma_{1}^{3}), \mu_{\disk}(\gamma_{2}^{3m_{2}}))=(3,5)$ or $(5,3)$ or $(3,3)$. 

Suppose that $(\mu_{\disk}(\gamma_{1}^{3}), \mu_{\disk}(\gamma_{2}^{3m_{2}}))=(3,5)$. Then $m_{2}=1$ or $2$ (Clam \ref{cla;essential}). If  $m_{2}=1$, $\alpha$ satisfies  Lemma \ref{lem:indechcomp} (5).  If  $m_{2}=2$, $ \mu_{\disk}(\gamma_{2}^{3m_{2}})=\mu_{\disk}((\gamma_{2}^{3})^{2}))=5$ implies that $\mu_{\disk}(\gamma_{3})=3$ (Claim \ref{cla;essential}) and thus  $\alpha=(\gamma_{1},1)\cup (\gamma_{2},2)$ satisfies Lemma \ref{lem:indechcomp} (3).

Suppose that $(\mu_{\disk}(\gamma_{1}^{3}), \mu_{\disk}(\gamma_{2}^{3m_{2}}))=(5,3)$ or $(3.3)$. $\mu_{\disk}(\gamma_{2}^{3m_{2}})=3$ implies that $m_{1}=1$ (Proposition \ref{conleycovering}) and thus $\alpha=(\gamma_{1},1)\cup (\gamma_{2},1)$ satisfies  (5) in Lemma \ref{lem:indechcomp}.

\item[(d) If $(k-k',k')=(2,0)$.]
In this case, we have $\alpha=(\gamma_{1},m_{1})\cup (\gamma_{2},m_{2})$ for elliptic orbits $\gamma_{1}$ and $\gamma_{2}$. It follows from (\ref{ineqkk'}) that  $ 10\geq \mu_{\disk}(\gamma_{1}^{3m_{1}})+\mu_{\disk}(\gamma_{2}^{3m_{2}})$. Without of loss generality, we may assume that $\mu_{\disk}(\gamma_{1})\geq \mu_{\disk}(\gamma_{2})$. Then  we have  $(\mu_{\disk}(\gamma_{1}^{3m_{1}}), \mu_{\disk}(\gamma_{2}^{3m_{2}}))=(7,3),(5,5),(5,3),(3,3)$. 

Suppose that $(\mu_{\disk}(\gamma_{1}^{3m_{1}}), \mu_{\disk}(\gamma_{2}^{3m_{2}}))=(7,3)$. Then $m_{2}=1$ and $\mu_{\disk}(\gamma_{2})=3$ (Proposition \ref{conleycovering}) and in addition either $m_{1}=1$ or $m_{1}=2$ or $m_{1}=3$ (Claim \ref{cla;essential}). If $m_{1}=1$, then $\alpha$ satisfies (5) in Lemma \ref{lem:indechcomp}. If $m_{1}=2$, then $\mu_{\disk}(\gamma_{2}^3)=3$ (Claim \ref{cla;essential}) and thus $\alpha$ satisfies (3) in Lemma \ref{lem:indechcomp}.  If $m_{1}=3$, then since $[\alpha]=m_{1}[\gamma_{1}]+[\gamma_{2}]=0$, $[\gamma_{2}]=0$ and thus $\gamma_{2}$ is contractible. But this is a contradiction because $\mu_{\disk}(\gamma_{2}^{3})\geq 2\times 3+1=7$ (Proposition \ref{conleycovering}).

Suppose that $(\mu_{\disk}(\gamma_{1}^{3m_{1}}), \mu_{\disk}(\gamma_{2}^{3m_{2}}))=(5,5)$. Then it follows from Claim \ref{cla;essential} that either $m_{i}=1$ or $2$ for each $i=1,2$ and in addition if $m_{i}=2$, then $\mu_{\disk}(\gamma_{i}^3)=3$. This means that  $\alpha$ satisfies either (4) or (5) or (6) in Lemma \ref{lem:indechcomp}. 

Suppose that $(\mu_{\disk}(\gamma_{1}^{3m_{1}}), \mu_{\disk}(\gamma_{2}^{3m_{2}}))=(5,3)$.  Then $m_{2}=1$ and $\mu_{\disk}(\gamma_{2})=3$ (Proposition \ref{conleycovering}) and in addition either $m_{1}=1$ or $m_{1}=2$. If $m_{1}=1$, then $\alpha$ satisfies (5) in Lemma \ref{lem:indechcomp}. If $m_{1}=2$, then $\mu_{\disk}(\gamma_{2}^3)=3$ (Claim \ref{cla;essential}) and thus $\alpha$ satisfies (3) in Lemma \ref{lem:indechcomp}.

Suppose that $(\mu_{\disk}(\gamma_{1}^{3m_{1}}), \mu_{\disk}(\gamma_{2}^{3m_{2}}))=(3,3)$. Then $m_{1}=m_{2}=1$. Thus $\alpha$ satisfies (5) in Lemma \ref{lem:indechcomp}.

\item[(e) If $(k-k',k')=(0,1)$.]
In this case, we have $\alpha=(\gamma_{1},1)$ for a negative hyperbolic $\gamma_{1}$ and thus  $\alpha$ satisfies  (7)  in Lemma \ref{lem:indechcomp}.
\item[(f) If $(k-k',k')=(1,0)$.]
In this case, $\alpha=(\gamma_{1},m_{1})$ for some $m_{1}$ and elliptic $\gamma$. Suppose that $\gamma_{1}$ is contractible. Then $\mu_{\disk}(\gamma_{1}^{m_{1}})\geq 2m_{1}+1$ (Proposition \ref{conleycovering}). If $m_{1}\geq 2$, we have $\mathrm{ind}(u)=\mu_{\disk}(\gamma_{1}^{m_{1}})-1\geq 4$. This contradicts $\mathrm{ind}(u)=2$. Therefore $m_{1}=1$ and $\alpha$ satisfies (7). 

Next, suppose that $\gamma_{1}$ is not contractible. Since $[\alpha]=0$, $m_{1}=3m_{1}'$ for some $m_{1}'\in \ZZ_{>0}$. Therefore, we have $\mathrm{ind}(u)=\mu_{\disk}((\gamma_{1}^3)^{m_{1}'})-1\geq 2m_{1}'$. This implies that $m_{1}'=1$ and $\mu_{\disk}(\gamma_{1}^3)=3$. Hence we have $\alpha=(\gamma_{1},3)$, which satisfies (2).

In summary. we complete the proof of Lemma \ref{lem:indechcomp}.
\end{proof}

\section{Rational open book decompositions and binding orbits}
In this section, we narrow down the list in Lemma \ref{lem:indechcomp}. For the purpose, we observe topological properties of the moduli space of holomorphic curves.  
For $u\in \mathcal{M}^{J}(\alpha,\emptyset)$, let $\mathcal{M}^{J}_{u}$ denote the connected component of $\mathcal{M}^{J}(\alpha,\emptyset)$ containing $u$. The next lemma plays important roles in what follows.
\begin{lem}\label{lem:openbook}
     Let $\alpha$ be an ECH generator with $\langle U_{J,z}\alpha,\emptyset \rangle\neq 0$. Suppose that $u\in \mathcal{M}^{J}(\alpha,\emptyset)$ is a $J$-holmorphic curve counted by $U_{J,z}$.  Then the quotient space $\mathcal{M}^{J}_{u}/\RR$ is compact and thus diffeomorphic to $S^{1}$.  In addition, for any section $s:\mathcal{M}^{J}_{u}/\RR \to \mathcal{M}^{J}_{u}$,  $\bigcup_{t\in S^{1}}\mathrm{\pi}(s(t))$ gives  a rational open book decomposition on $Y \cong L(3,1)$ where $\mathrm{\pi}:\RR \times Y \to Y$ is the projection.
\end{lem}
To prove Lemma \ref{lem:openbook}, we introduce the following proposition given in \cite{CHP}.
\begin{prp}[\cite{CHP}]\label{openbookhut}
 Let $(Y,\lambda)$ be a non-degenerate contact three-manifold, and let
$J$ be a compatible almost complex structure on $\RR \times Y$. Let $C$ be an irreducible
$J$-holomorphic curve in $\RR \times Y$ such that:
\item[(1).]  Every $C\in \mathcal{M}^{J}_{C}$ is embedded in $\RR \times Y$.
\item[(2).]  $C$ is of genus $0$, has no end asymptotic to a positive hyperbolic orbit and $\mathrm{ind}(C)=2$.
\item[(3).] $C$ does not have two positive ends, or two negative ends, at covers of the same simple Reeb orbit.
\item[(4).] Let $\gamma$ be a simple orbit with rotation nnumber $\theta \in \RR \backslash \QQ$.  If $C$ has a positive end at an $m$-fold cover of $\gamma$, then $\mathrm{gcd}(m,\lfloor m\theta \rfloor)=1$. If $C$ has a negative end at an $m$-fold cover of $\gamma$, then $\mathrm{gcd}(m,\lceil m\theta \rceil)=1$.
\item[(5).]$\mathcal{M}^{J}_{C}/\RR$ is compact.

Then $\pi(C)\subset Y$ is a global surface of section for the Reeb flow.  In addition, for any section $s:\mathcal{M}^{J}_{C}/\RR \to \mathcal{M}^{J}_{C}$,  $\bigcup_{t\in S^{1}}\pi(s(t))$ gives  a rational open book decomposition supporting the contact structure $\mathrm{Ker}\lambda=\xi$.
\end{prp}
\begin{proof}[\bf Proof of Lemma \ref{lem:openbook}]
It follows from Lemma \ref{lem:indechcomp} that any $u\in \mathcal{M}^{J}(\alpha,\emptyset)$ counted by $U_{J,z}$ satisfies (1), (2), (3) in Proposition \ref{openbookhut}. To prove that $u\in \mathcal{M}^{J}(\alpha,\emptyset)$ satisfies (4) in Proposition \ref{openbookhut}, we recall the following property.
\begin{cla}\label{gcd}
Let $\theta\in \RR \backslash \ZZ$. For any $q\in S_{-\theta}$ , $\mathrm{gcd}(q,\lfloor q\theta \rfloor)=1$.  
\end{cla}
\begin{proof}[\bf Proof of Claim \ref{gcd}]
Claim \ref{gcd} follows directly from the definition of $S_{\theta}$. Here, note that $-\lfloor\theta q \rfloor=\lceil -q \theta \rceil$.
\end{proof}
According to Proposition \ref{partitioncondition}, if an end  of $u\in \mathcal{M}^{J}(\alpha,\emptyset)$ is asymptotic to simple orbit $\gamma$ with some multiplicity, the multiplicity is in $S_{-\theta}$ where $\theta$ is the monodromy angle of $\gamma$. Therefore it follows that$u\in \mathcal{M}^{J}(\alpha,\emptyset)$ satisfies  (4) in Proposition \ref{openbookhut}.

At last, we check that $u\in \mathcal{M}^{J}(\alpha,\emptyset)$ satisfies (5) in Proposition \ref{openbookhut}.
Suppose that $\mathcal{M}^{J}_{u}/\mathbb{R}$ is not compact. Let $\overline{\mathcal{M}^{J}_{u}/\mathbb{R}}$ denote the compactified space  of $\mathcal{M}^{J}_{u}/\mathbb{R}$ in the sense of SFT compactness. Choose $\Bar{u}\in \overline{\mathcal{M}^{J}_{u}/\mathbb{R}} \backslash (\mathcal{M}^{J}_{u}/\mathbb{R})$. $\Bar{u}$ consists of some $J$-holomorphic curves in several floors. Let $u'$ be the component of $\Bar{u}$ in the lowest floor. Then there is an orbit set $\beta$ such that $u'\in \mathcal{M}^{J}(\beta,\emptyset)/\mathbb{R}$. By the additivity of ECH index, we have $I(\beta,\emptyset)=1$. This contradicts Lemma \ref{prp:immersed}. Thus $\mathcal{M}^{J}_{u}/\mathbb{R}$ is compact.    
\end{proof}

\begin{lem}\label{lem:ex1}
  $L(3,1)$ does not admit rational open book decompositions coming from  (6), (7) in Lemma \ref{lem:indechcomp}.
\end{lem}
\begin{proof}[\bf Proof of Lemma \ref{lem:ex1}]
    It is obvious that if a 3-manifol $Y$ has a open book decomposition such that each page is embedded disk, then $Y$ is diffeomorphic to $S^{3}$. Therefore (7) is excluded. 
    
Next, we consider (6) in Lemma \ref{lem:indechcomp}. For $\gamma_{1}$ and $\gamma_{2}$, we take Martinet tubes  $F_{i}:\RR/T_{\gamma_{i}}\ZZ \times \DD_{\delta} \to \Bar{U_{i}}$ for a sufficiently small $\delta>0$ where $\gamma_{i}\subset U_{i}$. Since $\pi(s(t))$ is embedded and connected on $Y$, $F_{i}^{-1}(\RR/T_{\gamma_{i}}\ZZ \times\partial \DD_{\delta} \cap \pi(s(t)))$ is $(2,p_{i})$-cable for odd integers $p_{i}\in \ZZ$ for any $t\in S^{1}$. In addition, the gluing map from $F_{1}(\RR/T_{\gamma_{1}}\ZZ \times\partial \DD_{\delta})$ to $F_{2}(\RR/T_{\gamma_{2}}\ZZ \times\partial \DD_{\delta})$ which maps the $(2,p_{1})$-cabling curve to the $(-2,-p_{2})$-cabling curve along $\mathrm{pr}_{2}(s(t))$ for each $t\in S^{1}$ recovers $Y \cong L(3,1)$ (note the sign). We note that the gluing map is described as 
\begin{equation*}
\begin{pmatrix}
a & -3 \\
c & d \\
\end{pmatrix}
\end{equation*}
in standard longitude-meridian coordinates on the torus. Since the matrix send a $(2,p_{1})$-cabling curve to a $(-2,-p_{2})$-cabling curve, it follows from the first line of the matrix that $-2=2a-3p_{1}$. Since $p_{1}$ is odd, this can not occur. Thus (7) is excluded.
\end{proof}
\begin{lem}\label{lem:ex2}
  Any open book decomposition coming from (5) in Lemma \ref{lem:indechcomp} does not support $(L(3,1),\xi_{\mathrm{std}})$.
\end{lem}
\begin{proof}[\bf Proof of Claim \ref{lem:ex2}]
    As mentioned, for any section $s:\mathcal{M}^{J}_{u}/\RR \to \mathcal{M}^{J}_{u}$,  $\bigcup_{t\in S^{1}} \pi(s(t)))$ gives  a rational open book decomposition of $L(3,1)$ supporting $\xi_{\mathrm{std}}$. But it is well-known that the contact structure on $L(3,1)$ supported by an open book decomposition such that each page is an embedded annulus is overtwisted. This completes the proof.
\end{proof}
\begin{lem}\label{lem;(1)}
 Suppose that  $\alpha$ satisfies either (1) or (2) or (3) or (4) in Lemma \ref{lem:indechcomp}. Then any simple orbit $\gamma_{i}$  in $\alpha$ is in $\mathcal{S}_{3}$.
\end{lem}
\begin{proof}[\bf Proof of Lemma \ref{lem;(1)}]
 It follows from Theorem \ref{fundament} that  if $\alpha$ satisfies (2) in Lemma \ref{lem:indechcomp}, $\gamma$ in $\alpha$ is in $\mathcal{S}_{3}$.

 Next, we consider (1),  (3), (4) in Lemma \ref{lem:indechcomp}.

    We note that each monodromy of  (rational) open book decomposition coming from (1),  (3), (4) is unique up to isotopy. Indeed, in the case of  (3) or (4), it follows  from straightforward arguments since each page is annulus type. In the case of (1), it is complicated but follows from the classification of contact structures supported by planer open book decompositions with three boundaries given in \cite{Ar}.
    Therefore, it  suffices to check the statement specifically.
    For the purpose, we give the specific  (rational) open book decompositions as Milnor fibrations. 
    
    Let $S^{3}:=\{(z_{1},z_{2})|\,\,|z_{1}|^2+|z_{2}|^{2}=1\,\,\} \subset \CC^{2} $ denote the unit sphere. Recall that $(S^{3},\lambda_{0}|_{S^{3}})$ is a contact 3-sphere with tight contact structure whose Reeb vector field is given by the derivative of the action of multiplying by $e^{2\pi t}$. In this case, any flow is periodic and any simple periodic orbit is a Hopf fiber. In addition, it is obvious that any Hopf fiber is in $\mathcal{S}_{1}$.

    Let  $(L(3,1),\lambda_{0}|_{L(3,1)})$ denote the contact manifold obtained by taking the quotient of $(S^{3},\lambda_{0}|_{S^{3}})$ under the action  $(z_{1},z_{2}) \to e^{\frac{2\pi i}{3}}(z_{1},z_{2})$.
   Under the action, any Hopf fiber is preserved.  Therefore any flow on $(L(3,1),\lambda_{0}|_{L(3,1)})$ is periodic.  In addition, any simple periodic orbit is tangent to a image of a Hopf fiber of $S^{3}\to L(3,1)$ and thus obviously in $\mathcal{S}_{3}$.

 At first, we describe   (1) as a Milnor fibration. 
 Consider a map $f:\CC^2\to S^{1}$ with $f(z_{1},z_{2})=z_{1}^{3}+z_{2}^{3}$. It is easy to check that $g=f/|f|:S^{3}\backslash f^{-1}(0)\to S^{1}$ gives a open book decomposition supporting the tight contact structure and the binding consists of periodic orbits of $(S^{3},\lambda_{0}|_{S^{3}})$. In addition, for each $\theta \in S^{1}$,  $g^{-1}(\theta)$ is connected and of genus 1 with three boundary components.  Since $f(e^{\frac{2\pi}{3}}z_{1},e^{\frac{2\pi}{3}}z_{2})=f(z_{1},z_{2})$, $g$ induces, by taking a quotient space, an open book decomposition on $(L(3,1),\lambda_{0}|_{L(3,1)})$. It  follows  that each fiber of the open book decomposition is of genus 0 and three boundary components consisting of periodic orbits on $(L(3,1),\lambda_{0}|_{L(3,1)})$, which gives a description of (1).  Therefore  any simple orbit $\gamma_{i}$  in $\alpha$ is in $\mathcal{S}_{3}$.

 Second we describe (3) and (4).  Consider a map $f:\CC^2\to S^{1}$ with $f(z_{1},z_{2})=z_{1}z_{2}^{2}$.  $g=f/|f|:S^{3}\backslash f^{-1}(0)\to S^{1}$ gives a open book decomposition supporting the tight contact structure and the binding consists of periodic orbits of $(S^{3},\lambda_{0}|_{S^{3}})$ (see the remark below). In addition,  since $f(e^{\frac{2\pi i}{3}}z_{1},e^{\frac{2\pi i}{3}}z_{2})=f(z_{1},z_{2})$, $g$ induces, by taking a quotient space, an open book decomposition on $(L(3,1),\lambda_{0}|_{L(3,1)})$ whose boundary consists of periodic orbits on $(L(3,1),\lambda_{0}|_{L(3,1)})$. 
 Moreover, for any $\theta\in S^{1}$, $g^{-1}(\theta)$ is of genus $0$ with two boundary components  and thus each page of  the induced open book decomposition of $(L(3,1),\xi_{\mathrm{std}})$ is of genus $0$ with two boundary components which are periodic orbits of $(L(3,1),\lambda_{0}|_{L(3,1)})$, which gives a description of (3) and (4). This means that if   $\alpha$ satisfies  (3) or (4)  in Lemma \ref{lem:indechcomp}, then any simple orbit $\gamma_{i}$  in $\alpha$ is in $\mathcal{S}_{3}$.
\end{proof}
\begin{rem}
Consider a map $g=f/|f|:S^{3}\to S^{1}$ with $f(z_{1},z_{2})=\frac{z_{1}z_{2}^{2}}{|z_{1}z_{2}^{2}|}$. Then for $\theta \in S^{1}$, the fiber $g^{-1}(\theta)$ is the set $\{(re^{2\pi i \theta_{1}},\sqrt{1-r^2}e^{2\pi i \theta_{1}})|\,\theta_{1}+2\theta_{2}=\theta\,\}$.
\end{rem}
\begin{rem}
    Let $p_{1}:M\backslash B_{1} \to S^{1}$ and $p_{2}:M\backslash B_{2} \to S^{1}$ be rational open book decompositions supporting a contact manifold $(M,\xi)$. Suppose that each page are diffeomorphic to each other and the monodromies are the same. Then the binding links $B_{1}$ and $B_{2}$ are the same as transversal links. Indeed, Let $B_{i}\subset U_{i}$ be a sufficiently small tubular neighborhood. Then we can construct a diffeomorphism $f:M \to M$ such that $f(B(1))=B(2)$ and $f$ maps each page of $p_{1}:M\backslash U_{1}\to S^{1}$ to $p_{2}:M\backslash U_{2}\to S^{1}$. It follows from standard arguments that we can construct an isotopy of contact structures from $f_{*}(\xi)$ to $\xi$ so that the binding link $f(B(1))=B(2)$ is preserved with transversal to the isotopy contact structures.
\end{rem}
\begin{proof}[\bf Proof of Proposition \ref{main;prop}]
Having Lemma \ref{lem:ex1}, Lemma \ref{lem:ex2} and Lemma \ref{lem;(1)},   it suffices to exclude (4) in Lemma \ref{lem:indechcomp}. Suppose that $\alpha$ satisfies (4) in Lemma \ref{lem;(1)}. Then there are two elliptic orbits $\gamma_{1}$ and $\gamma_{2}$ such that $\alpha=(\gamma_{1},2)\cup (\gamma_{2},1)$,  $\mu_{\disk}(\gamma_{1}^{6})=5$ and  $\mu_{\disk}(\gamma_{2}^{3})=5$.
Note that $\mu_{\disk}(\gamma_{1}^{6})=5$ means $\mu_{\disk}(\gamma_{1}^3)=3$

In addition, according to Lemma \ref{lem;(1)}, we have $\gamma_{2}\in \mathcal{S}_{3}$. Recall that $\gamma_{2}\in \mathcal{S}_{3}$ means that there is a rational Seifert surface $u:\DD \to L(3,1)$ with $u(e^{2\pi t})=\gamma_{2}(3T_{\gamma_{2}}t)$ and $sl_{\xi}^{\QQ}(\gamma_{2})=-\frac{1}{3}$. 
Take a Martinet tube $F:\RR/T_{\gamma_{2}}\ZZ \times \DD_{\delta} \to \Bar{U}$ for a sufficiently small $\delta>0$ onto a small  open neighbourhood $\gamma_{2} \subset \Bar{U}$.
Let  $\tau_{F}:\gamma_{i}^{*}\xi \to \RR/T_{\gamma_{2}}\ZZ \times \RR^{2}$ be a trivialization induced by $F$. We may choose $F$ so that  $\mu_{\tau_{F}}(\gamma_{2})=1$.

Note that $F^{-1}(u(\DD)\cap\partial \Bar{U})$ is a $(3,r)$ cable such that $r=3k-1$ for some $k\in \ZZ$ with respect to the coordinate of $\RR/T_{\gamma_{2}}\ZZ \times \DD $.
It follows from Lemma \ref{lem;key} that
\begin{equation*}
    \mu_{\tau_{F}}(\gamma_{2}^{3})-2r-6sl_{\xi}^{\QQ}(\gamma_{2})=\mu_{\disk}(\gamma_{2}^{3})
\end{equation*}
In this case, it follows from $r=3k-1$ and $sl_{\xi}^{\QQ}(\gamma_{2})=-\frac{1}{3}$ that  $\mu_{\tau_{F}}(\gamma_{2}^{3})-2r-6sl_{\xi}^{\QQ}(\gamma_{2})=\mu_{\tau_{0}}(\gamma_{2}^{3})+4-6k$. Since $\gamma_{2}$ is elliptic, there is $\theta \in \RR \backslash \QQ$ such that $\mu_{\tau_{F}}(\gamma_{2}^m)=2\lfloor m\theta \rfloor+1$ for every $m \in \ZZ_{>0}$. Recall that we take $\tau_{F}$ so that $\mu_{\tau_{F}}(\gamma_{2})=1$. Hence $0< \theta <1$ and so we have $\mu_{\tau_{0}}(\gamma_{2}^3)=1$ or  $3$ or $5$.

Since $\mu_{\disk}(\gamma_{2}^3)=5$. $\mu_{\tau_{0}}(\gamma_{2}^{3})+4-6k=\mu_{\disk}(\gamma^{3})$ holds only if $k=0$ and  $\mu_{\tau_{0}}(\gamma_{2}^3)=1$. Let $u\in \mathcal{M}^{J}(\alpha,\emptyset)$. From (\ref{eq:ineq}), it follows that
\begin{equation*}
    \begin{split}
    3\ind(u)=&-3(2-2g(u)-2)+2c_{\tau}(\xi|_{3[u]})+ 3\mu_{\tau_{0}}(\gamma_{2})+3\mu_{\tau}(\gamma_{1}^{2})\\
      =&-6+\mu_{\mathrm{disk}}(\gamma_{1}^{3})+\mu_{\mathrm{disk}}(\gamma_{2}^{6})\\&+(3\mu_{\tau_{0}}(\gamma_{2})-\mu_{\tau_{0}}(\gamma_{2}^{3})+3)+(3\mu_{\tau}(\gamma_{1}^2)-\mu_{\tau}((\gamma_{1}^{2})^3)+3)\\
      \geq&-6+\mu_{\mathrm{disk}}(\gamma_{1}^{3})+\mu_{\mathrm{disk}}(\gamma_{2}^{6})+(3\mu_{\tau_{0}}(\gamma_{2})-\mu_{\tau_{0}}(\gamma_{2}^{3})+3)+1
      \end{split}
\end{equation*}
Here $\tau$ is a trivialization on $\gamma_{1}$ and we use Lemma \ref{lem:iterate}. Since $\mu_{\tau_{F}}(\gamma_{2})=\mu_{\tau_{0}}(\gamma_{2}^3)=1$, we have $3\mu_{\tau_{0}}(\gamma_{2})-\mu_{\tau_{0}}(\gamma_{2}^{3})+3=5$. In addition since $\mu_{\mathrm{disk}}(\gamma_{1}^{3})=\mu_{\mathrm{disk}}(\gamma_{2}^{6})=5$, we have $\mu_{\mathrm{disk}}(\gamma_{1}^{3})+\mu_{\mathrm{disk}}(\gamma_{2}^{6})=10$. In summary we have 
\begin{equation*}
    -6+\mu_{\mathrm{disk}}(\gamma_{1}^{3})+\mu_{\mathrm{disk}}(\gamma_{2}^{6})+(3\mu_{\tau_{0}}(\gamma_{2})-\mu_{\tau_{0}}(\gamma_{2}^{3})+3)+1=10.
\end{equation*}
But $3\ind(u)=6$. This contradicts the inequality. This means that (4) in Lemma \ref{lem:indechcomp} is impossible. This completes the proof.
\end{proof}

\section{Proof of the main theorem under non-degeneracy}
What follows in this section is written in almost the same way with \cite[Section 4]{Shi},  but it's more detailed.

Let  $\langle \alpha_{1}+...+\alpha_{k} \rangle$ denotes the element in $ECH(Y,\lambda)$ for a sum of ECH generators  $\alpha_{1}+...+\alpha_{k}$  with $\partial_{J}(\alpha_{1}+...+\alpha_{k})=0$.

\begin{lem}\label{existenceofhol}
 Let $(L(3,1),\lambda)$ be a dynamically convex non-degenerate contact manifold with $\lambda \wedge d\lambda>0$. Let $\alpha_{1},...,\alpha_{k}$ be ECH generators with $[\alpha_{i}]=0$ and $I(\alpha_{i},\emptyset)=2$ for $i=1,...k$. Suppose that $\partial_{J}(\alpha_{1}+...+\alpha_{k})=0$ and $0 \neq \langle \alpha_{1}+...+\alpha_{k} \rangle 
\in  ECH_{2}((3,1),\lambda,0)$. Then there exists $i$ such that $\langle U_{J,z}\alpha_{i},\emptyset \rangle\neq  0$.
\end{lem}
\begin{proof}[\bf Proof of Lemma \ref{existenceofhol}]
The proof is the same with \cite[Lemma 4.6]{Shi}.
\end{proof}
\begin{lem}\label{lem:ECHgen}
  Let $(L(p,1),\lambda)$ be a (not 
necessarily dynamically convex) non-degenerate contact manifold with $\lambda \wedge d\lambda>0$. Let $\alpha_{\gamma}=(\gamma,p)$ for  $\gamma \in \mathcal{S}_{p}$.  If $\mu_{\disk}(\gamma^{p})=3$, then $\alpha_{\gamma}$ is an ECH generator. In addition $I(\alpha_{\gamma},\emptyset)=2$.
\end{lem}
\begin{proof}[\bf Proof of Lemma \ref{lem:ECHgen}]
According to Corollary \ref{main;cor},  $\gamma$ is elliptic and hence by definition $\alpha_{\gamma}=(\gamma,p)$ is an ECH generator.

    Take a Martinet tube $F:\RR/T_{\gamma}\ZZ \times \DD_{\delta} \to \Bar{U}$ for a sufficiently small $\delta>0$ onto a small  open neighbourhood $\gamma \subset \Bar{U}$. 

     Recall that $\gamma\in \mathcal{S}_{p}$ means that there is a rational Seifert surface $u:\DD \to L(p,1)$ with $u(e^{2\pi t})=\gamma(pT_{\gamma_{2}}t)$ and $sl_{\xi}^{\QQ}(\gamma)=-\frac{1}{p}$.
     
      We may take $F$ so that $F^{-1}(u(\DD)\cap\partial \Bar{U})$ is $(p,p-1)$-cabling such that $r=3k-1$ for some $k\in \ZZ$ with respect to the coordinate of $\RR/T_{\gamma_{2}}\ZZ \times \DD_{\delta} $. 
      Let  $\tau_{F}:\gamma^{*}\xi \to \RR/T_{\gamma}\ZZ \times \RR^{2}$ be the trivialization induced by $F$. Let $\theta \in \RR \backslash \QQ$ denote the monodoromy angle of $\gamma$ with respect to $\tau_{F}$. In this case,  $\mu_{\tau_{F}}(\gamma^{k})=2\lfloor k \theta \rfloor+1$ for any $k\in \ZZ_{>0}$. It follows from Lemma \ref{lem;key} that
      \begin{equation*}
    \mu_{\tau_{F}}(\gamma^{p})+4-2p=2\lfloor p \theta \rfloor +5-2p=\mu_{\mathrm{disk}}(\gamma^{p})=3.
      \end{equation*}
      and hence
      \begin{equation*}
          \lfloor p \theta \rfloor=p-1.
      \end{equation*}
      This implies that $1-\frac{1}{p}<\theta<1$. In particular, this means that $\mu_{\tau_{F}}(\gamma^k)=2(k-1)+1$ for any $1 \leq k \leq p$. Therefore we have
      \begin{equation*}
          \sum_{1 \leq k \leq p}\mu_{\tau_{F}}(\gamma^{k})=p^2.
      \end{equation*}

      To compute the relative self intersection number $Q_{\tau_{0}}$, take another  Seifert surface $u':\DD \to Y$ of $\gamma$  so that $u(\mathring{\DD})\cap u'(\mathring{\DD})=\emptyset$. This is possible because $u(\DD)$ is a page of  a rational open book decomposition. Now, we take immersed $S_{1},S_{2}\subset [0,1]\times Y$ which are transverse to $\{0,1\}\times Y$ so that $\mathrm{pr}_{2}(S_{1})=u(\DD)$ and $\mathrm{pr}_{2}(S_{2})=u'(\DD)$ where $\mathrm{pr}_{2}:[0,1]\times Y \to Y$ is the projection.
      Set $\{Z\}:=H_{2}(Y,\alpha_{\gamma},\emptyset)$.
       It follows from the definition that we have $Q_{\tau_{F}}(Z)=-l_{\tau_{F}}(S_{1},S_{2})+\#(S_{1}\cap S_{2})$ where $l_{\tau_{F}}(S_{1},S_{2})$ is the linking number.
       More precisely, $l_{\tau_{F}}(S_{1},S_{2})$ is  one half the signed number of crossings of $F^{-1}(\mathrm{pr}_{2}(S_{1}\cap\{1-\epsilon\}\times Y))$ with $F^{-1}(\mathrm{pr}_{2}(S_{2}\cap\{1-\epsilon\}\times Y))$ for a small $\epsilon>0$ in the projection $(\mathrm{id},\mathrm{pr}_{1}):\RR/T_{\gamma}\ZZ\times \RR^{2}\to \RR/T_{\gamma}\ZZ\times \RR$ where we naturally assume that $\RR/T_{\gamma}\ZZ\times \DD_{\delta}\subset \RR/T_{\gamma}\ZZ\times \RR^{2}$ and  $\mathrm{pr}_{1}:\RR^{2}\to \RR$ is the projection to the first coordinate.  For more details, see  \cite{H1}. 

       In this situation,  it follows from the choice of $S_{1}$ and $S_{2}$ that  $\#(S_{1}\cap S_{2})=0$. In addition, since we take   $F$ so that $F^{-1}(u(\DD)\cap\partial \Bar{U})$ is a $(p,p-1)$ cabling, we have $l_{\tau_{F}}(S_{1},S_{2})=p(p-1)$. In summary, we have $Q_{\tau_{F}}(Z)=-p(p-1)$.

      At last, we compute the relative first chern number $c_{\tau_{F}}$.  Since   $\tau_{\disk}:(\gamma^p)^{*}\xi \to \RR/pT_{\gamma}\ZZ \times \RR^2$ extends to a trivialization by definition, we have $c_{\tau_{0}}=\mathrm{wind}(\tau_{0},\tau_{\disk})$. It follows from (\ref{wind}) that $\mathrm{wind}(\tau_{0},\tau_{\disk})=2-p$ and hence $c_{\tau_{0}}=\mathrm{wind}(\tau_{0},\tau_{\disk})=2-p$.
      By summarizing above results, we have $I(\alpha_{\gamma},\emptyset)=(2-p)+(-p(p-1))+p^2=2$. This completes the proof.
      \end{proof}
\begin{lem}\label{lem:boundaryop}
 Let $(L(p,1),\lambda)$ be a  non-degenerate dynamically convex contact manifold with $\lambda \wedge d\lambda>0$.  Let $\alpha_{\gamma}=(\gamma,p)$ for  $\gamma \in \mathcal{S}_{p}$.  If $\mu_{\disk}(\gamma^{p})=3$, then
     there is no somewhere injective $J$-holomorphic curve satisfying the following;
    \item[(1).] There is only one positive end. In addition, the positive end is asymptotic to $\gamma$ with multiplicity $p$.
    \item [(2).] There is at least one negative end.
    \item[(3).] Any puncture on the domain is either positive or negative end.
\end{lem}
\begin{proof}[\bf Proof of Lemma \ref{lem:boundaryop}]
In this proof, we  set $Y=L(p,1)$.
    Suppose that $h:(\Sigma,j)\to (\RR \times Y,J)$ is  a somewhere injective curve satisfying the properties.
      
Recall that $\gamma\in \mathcal{S}_{p}$ means that there is a rational Seifert surface $u:\DD \to L(3,1)$ with $u(e^{2\pi t})=\gamma(pT_{\gamma_{2}}t)$ and $sl_{\xi}^{\QQ}(\gamma)=-\frac{1}{p}$ such that $u(\DD)$ is a Birkhoff section for $X_{\lambda}$ of disk type. We note that $H_{1}(Y\backslash \gamma) \cong \ZZ$. 

For a sufficiently large $s>>0$, consider $\pi(h(\Sigma)\cap ([-s,s]\times Y))\subset Y$.  Since $\#(\gamma \cap \pi(h(\Sigma)\cap ([-s,s]\times Y)))\geq 0$ because of positivity of intersection, it follows topologically that $\#(u(\DD) \cap \pi(h(\Sigma)\cap (\{s\}\times Y)))\geq \#(u(\DD) \cap \pi(h(\Sigma)\cap (\{-s\}\times Y)))$.  This contradicts  the next claim.
\begin{cla}\label{cla:wind}
   \item[(1).]  $0\geq \#(u(\DD) \cap \pi(h(\Sigma)\cap (\{s\}\times Y)))$
\item[(2).] $\#(u(\DD) \cap \pi(h(\Sigma)\cap (\{-s\}\times Y)))\geq 1$
 
\end{cla}
\begin{proof}[\bf Proof of Claim \ref{cla:wind}]
 Take  a Martinet tube $F:\RR/T_{\gamma}\ZZ \times \DD_{\delta} \to \Bar{U}$ for a sufficiently small $\delta>0$ onto a small  open neighbourhood $\gamma \subset \Bar{U}$. To prove Claim \ref{cla:wind}, we recall the following property.
 \begin{cla}
     Let $g:([0,\infty)\times S^{1},j_{0}) \to (\RR \times Y,J)$ be a $J$-holomorphic curve asymptotic to $\gamma$ with multiplicity $m$ where $S^{1}=\RR/\ZZ$ and $j_{0}$ is the standard complex structure. Let $\mathrm{pr}_{\DD_{\delta}}$ denote the projection  $\RR/T_{\gamma}\ZZ \times \DD_{\delta}\to \DD_{\delta}$. Define  $w(g)$ by the winding number of $\RR/\ZZ \to \DD_{\delta}$ which maps $t\in \RR /\ZZ \to \mathrm{pr}_{\DD_{\delta}}\circ F^{-1}(g(s,t))$. Then
     \begin{equation*}
         w(g)\leq \lfloor \frac{\mu_{\tau_{F}}(\gamma^{m}) }{2} \rfloor. 
     \end{equation*}
 \end{cla}
 The above claim follows immediately  from the properties of the Conley-Zehnder index and winding number of eigenfunctions of a certain self-adjoint operator. See \cite{HWZ1} and \cite{HWZ4}
     
We may take $F$ so that $F^{-1}( u (\DD)\cap\partial \Bar{U})$ is a $(p,p-1)$ cable. 
Let  $\tau_{F}:\gamma^{*}\xi \to \RR/T_{\gamma}\ZZ \times \RR^{2}$ be a trivialization induced by $F$. Let $\theta \in \RR \backslash \QQ$ denote the monodromy angle of $\gamma$ with respect to $\tau_{0}$. That is $\mu_{\tau_{0}}(\gamma^{k})=2\lfloor k \theta \rfloor+1$ for any $k\in \ZZ_{>0}$. It follows from the same argument that  $1-\frac{1}{p}<\theta<1$. In particular, this means that $\lfloor k \theta \rfloor=k-1$ for any $1 \leq k \leq p$. 

At first, we consider $F^{-1}(\pi(h(\Sigma)\cap (\{s\}\times Y))\cap\partial \Bar{U})$. Let $\mathrm{pr}_{\DD_{\delta}}:\RR / T_{\gamma} \ZZ \times \DD_{\delta}\to \DD_{\delta}$ denote the projection. Then it follows that the winding number of $\mathrm{pr}_{\DD_{\delta}} \circ F^{-1}(\pi(h(\Sigma)\cap (\{s\}\times Y))\cap\partial \Bar{U})\subset\DD_{\delta}$ is at most $\lfloor \frac{\mu_{\tau_{0}}(\gamma^p)}{2} \rfloor=p-1$. This means that the slope of $F^{-1}(\pi(h(\Sigma)\cap (\{s\}\times Y))\cap\partial \Bar{U})$ is at most $\frac{p-1}{p}$.
Since the multiplicity of the positive end is $p$, it follows  that $0\geq \#(u(\DD) \cap \pi(h(\Sigma)\cap (\{s\}\times Y)))$. This proves (1).

Next, we consider $F^{-1}(\pi(h(\Sigma)\cap (\{-s\}\times Y))\cap\partial \Bar{U})$. We note that the total multiplicity of negative ends asymptotic to $\gamma$ is at most $p-1$ because of the Stokes' theorem. Suppose that a negative end of $h$ is asymptotic to $\gamma$ with multiplicity $1\leq k \leq p-1$. Then it follows that the winding number of $\mathrm{pr}_{\DD_{\delta}} \circ F^{-1}(\pi(h(\Sigma)\cap (\{s\}\times Y))\cap\partial \Bar{U})\subset\DD_{\delta}$ is at least $\lceil\frac{\mu_{\tau_{0}}(\gamma^k)}{2} \rceil=k$.
This means that the slope of $F^{-1}(\pi(h(\Sigma)\cap (\{s\}\times Y))\cap\partial \Bar{U})$ is at least $1$. On the other hand,   $F^{-1}( u (\DD)\cap\partial \Bar{U})$ is a $(p,p-1)$ cable and hence the slope is $\frac{p-1}{p}$. This means that the negative end intersects $u(\DD)$ positively. Therefore, if there is a negative end asymptotic to $\gamma$, we have $\#(u(\DD) \cap \pi(h(\Sigma)\cap (\{-s\}\times Y)))\geq 1$. 
Finally suppose that there is no negative end asymptotic to $\gamma$. Since $u(\DD)$ is a Birkhoff section for $X_{\lambda}$,  any periodic orbit other than $\gamma$ intersects $u(\DD)$ positively. Therefore it follows from the assumption that $\#(u(\DD) \cap \pi(h(\Sigma)\cap (\{-s\}\times Y)))\geq 1$. This completes the proof.
\end{proof}
Having Claim \ref{cla:wind}, Lemma \ref{lem:boundaryop} follows.
\end{proof}
\begin{lem}\label{lem:boundvanish}
    For any $\gamma \in \mathcal{S}_{p}$ with $\mu_{\disk}(\gamma^{p})=3$, $\partial_{J} \alpha_{\gamma} =0$.
\end{lem}
\begin{proof}[\bf Proof of Lemma \ref{lem:boundvanish}]
    Let $\beta$ be an ECH generator with $I(\alpha_{\gamma},\beta)$. Let $h \in \mathcal{M}^{J}(\alpha_{\gamma},\beta)$. It follows from the partition condition that the number of  positive ends of $h$ is one and the positive end is  asymptotic to $\gamma$ with multiplicity $p$. According to Lemma \ref{lem:boundaryop}, such a $h$ does not exist. This means that $\mathcal{M}^{J}(\alpha_{\gamma},\beta) = \emptyset$. Hence we have $\partial_{J} \alpha_{\gamma} =0$.
\end{proof}
Now, we focus on $Y \cong L(3,1)$. Since $\partial_{J} \alpha_{\gamma} =0$ for any $\gamma \in \mathcal{S}_{3}$ with $\mu_{\disk}(\gamma^3)=3$, we may consider $\alpha_{\gamma}$ as an element in $ECH_{2}(Y,\lambda,0)$. The following lemma means that $\alpha_{\gamma}$ is not zero in $ECH_{2}(Y,\lambda,0)$.

\begin{lem}\label{nonzero}
     For any $\gamma \in \mathcal{S}_{3}$ with $\mu_{\disk}(\gamma^3)=3$,  $0\neq \langle \alpha_{\gamma} \rangle = \langle(\gamma,3) \rangle \in ECH_{2}(Y,\lambda,0)$.
\end{lem}

To prove Lemma \ref{nonzero}, define a set $\mathcal{G}$ consisting of ECH generators as
\begin{equation}
    \mathcal{G}:=\{\alpha|\,\,\langle U_{J,z} \alpha,\emptyset \rangle\neq 0 \,\, \}.
\end{equation}
Note that $\langle U_{J,z} \alpha,\emptyset \rangle\neq 0 $  if and only if  $\alpha \in \mathcal{G}$.
\begin{cla}\label{inG}
     For any $\gamma \in \mathcal{S}_{3}$ with $\mu_{\disk}(\gamma^3)=3$, $\alpha_{\gamma}\in \mathcal{G}$.
\end{cla}
\begin{proof}[\bf Proof of Claim \ref{inG}]
While this proof is completely the same with \cite[Lemma 4.6]{Shi}, we give it below.

Recall that each page of the rational open book decomposition constructed in Theorem \ref{fundament}( \cite{HrS}) is the projection of $J$-holomorphic curve from $(\mathbb{C},i)$ to $L(3,1)$. Moreover in this case,  $:\mathcal{M}^{J}(\alpha_{\gamma},\emptyset)/\mathbb{R}$ is compact and  any two distinct elements $u_{1},u_{2}\in \mathcal{M}^{J}(\alpha_{\gamma},\emptyset)$ has no intersection point. 
Hence $:\mathcal{M}^{J}(\alpha_{\gamma},\emptyset)/\mathbb{R}\cong S^{1}$ and
     for a section $s:\mathcal{M}^{J}(\alpha_{\gamma},\emptyset)/\mathbb{R}\to \mathcal{M}^{J}(\alpha_{\gamma},\emptyset)$, $\bigcup_{\tau\in \mathcal{M}^{J}(\alpha_{\gamma},\emptyset)/\mathbb{R}}\overline{\pi(s(\tau))} \to \mathcal{M}^{J}(\alpha_{\gamma},\emptyset)/\mathbb{R}$ is an (rational) open book decomposition of $L(2,1)$. This implies that for $z\in L(3,1)$ not on $\gamma$, there is exactly one $J$-holomorphic curve in $ \mathcal{M}^{J}(\alpha_{\gamma},\emptyset)$ through $(0,z)\in \mathbb{R}\times L(3,1)$. Therefore we have $\langle U_{J,z} \alpha_{\gamma},\emptyset \rangle\neq 0 $.
\end{proof}
\begin{cla}\label{index1bound}
Let $\alpha \in \mathcal{G}$. Suppose that $\beta$ is an ECH generator with $I(\beta,\alpha)=1$.
Then 
\begin{equation}
    \sum_{\alpha\in \mathcal{G}}\langle \partial_{J}\beta,\alpha \rangle=0
\end{equation}
\end{cla}
\begin{proof}[\bf Proof of Claim \ref{index1bound}]

As the same with Claim \ref{inG},  this proof is completely the same with \cite[Lemma 4.7]{Shi}. We give it below.

Write
\begin{equation}
    \partial_{J}\beta=\sum_{\alpha\in \mathcal{G}}\langle \partial_{J}\beta,\alpha \rangle \alpha+\sum_{I(\beta,\sigma)=1,\sigma\notin \mathcal{G}}\langle \partial_{J}\beta,\sigma \rangle \sigma.
\end{equation}
Then we have
\begin{equation}
   \langle U_{J,z}\partial_{J}\beta,\emptyset \rangle=\sum_{\alpha\in \mathcal{G}}\langle \partial_{J}\beta,\alpha \rangle \langle U_{J,z} \alpha,\emptyset \rangle +\sum_{I(\beta,\sigma)=1,\sigma\notin \mathcal{G}}\langle \partial_{J}\beta,\sigma \rangle \langle \sigma,\emptyset \rangle=\sum_{\alpha\in \mathcal{G}}\langle \partial_{J}\beta,\alpha \rangle
\end{equation}
Here we use that for $\alpha\in \mathcal{G}$, $\langle U_{J,z} \alpha,\emptyset \rangle=1$ and for $\sigma$ with $\sigma \notin \mathcal{G}$, $\langle U_{J,z} \sigma,\emptyset \rangle=0$.
Since $U_{J,z}\partial_{J}=\partial_{J}U_{J,z}$, we have $\langle U_{J,z}\partial_{J}\beta,\emptyset \rangle=\langle \partial_{J}U_{J,z}\beta,\emptyset \rangle=0$. This completes the proof.
\end{proof}

\begin{proof}[\bf Proof of Lemma \ref{nonzero}]
Suppose that $0=\langle \alpha_{\gamma} \rangle  \in ECH_{2}(Y,\lambda,0)$.
Then there are ECH generators $\beta_{1},...\beta_{j}$ with $I(\beta_{i},\alpha_{\gamma})=1$ for any $i$ such that $\partial_{J}(\beta_{1}+...+\beta_{j})=\alpha_{\gamma}$. From Lemma \ref{index1bound}, we have 
\begin{equation}
    \sum_{1\leq i \leq j} \sum_{\alpha\in \mathcal{G}}\langle \partial_{J}\beta_{i},\alpha \rangle=\sum_{\alpha\in \mathcal{G}}\langle \alpha_{\gamma},\alpha \rangle=0.
\end{equation}
But since $\alpha_{\gamma}\in \mathcal{G}$,  $\sum_{\alpha\in \mathcal{G}}\langle \alpha_{\gamma},\alpha \rangle =1$. This is a contradiction. We complete the proof.
\end{proof}
\begin{proof}[\bf Proof of Theorem \ref{maintheorem} under non-degeneracy]

 At first, we estimate $c_{1}^{\mathrm{Ech}}(L(3,1)\lambda)$ from below.   It follows from the definitin of ECH spectrum that we can take an ECH generator $\alpha$ such that $\langle U_{J,z}\alpha,\emptyset \rangle \neq 0$ and $A(\alpha)\leq c_{1}^{\mathrm{ECH}}(L(3,1),\lambda)$. According to Proposition \ref{main;prop}, $\alpha$ contains  $\gamma \in \mathcal{S}_{3}$ such that $\mu_{\disk}(\gamma^3)=3$. We may assume that $\gamma$ has the minimum period in $\alpha$. Then $\int_{\gamma}\lambda\leq \frac{1}{3}A(\alpha)\leq \frac{1}{3}c_{1}^{\mathrm{ECH}}(L(3,1),\lambda)$. This means that there is  a $\gamma \in \mathcal{S}_{3}$ with $\mu_{\disk}(\gamma^3)=3$ such that  $\int_{\gamma}\lambda\leq \frac{1}{3}A(\alpha)\leq \frac{1}{3}c_{1}^{\mathrm{ECH}}(L(3,1),\lambda)$.

 At last, we estimate $c_{1}^{\mathrm{EcH}}(L(3,1)\lambda)$ from above. 
 Since  $0\neq \langle \alpha_{\gamma} \rangle = \langle(\gamma,3) \rangle \in ECH_{2}(Y,\lambda,0)$ for any $\gamma \in \mathcal{S}_{3}$ with $\mu_{\disk}(\gamma^3)=3$ (Lemma \ref{nonzero}),  we have $c_{1}^{\mathrm{EcH}}(L(3,1)\lambda)\leq A(\alpha_{\gamma})$ for any $\gamma \in \mathcal{S}_{3}$ with $\mu_{\disk}(\gamma^3)=3$. This means that  $ \frac{1}{3}c_{1}^{\mathrm{ECH}}(L(3,1)\lambda)\leq \inf_{\gamma\in \mathcal{S}_{3},\,\,\mu_{\disk}(\gamma^3)=3}\int_{\gamma}\lambda$. 

 In summary, we have $ \frac{1}{3}c_{1}^{\mathrm{ECH}}(L(3,1)\lambda)=\inf_{\gamma\in \mathcal{S}_{3},\,\mu_{\disk}(\gamma^3)=3}\,\int_{\gamma}\lambda$.
\end{proof}
\section{Extend the results to degenerate cases}
In this subsection, we extend the above result to degenerate case as a limit of non-degenerate result.
The content in this section is completely the same with \cite[\S 5]{Shi}. However we provide the details as follows.

At first, we show;
\begin{prp}\label{degeneratel2}
     Assume that $(L(3,1),\lambda)$ is strictly convex. Then there exists a simple orbit $\gamma\in \mathcal{S}_{3}$ such that $\mu_{\disk}(\gamma^{3})=3$ and $\int_{\gamma}\lambda = \frac{1}{3}\,c_{1}^{\mathrm{ECH}}(L(3,1),\lambda)$. In particular,
    \begin{equation}
   \inf_{\gamma\in \mathcal{S}_{3},\mu_{\disk}(\gamma^{3})=1}\int_{\gamma}\lambda \leq \frac{1}{3}\,c_{1}^{\mathrm{ECH}}(L(3,1),\lambda).
\end{equation}
\end{prp}
\begin{proof}[\bf Proof of Proposition \ref{degeneratel2}]
    Let $L=c^{\mathrm{ECH}}_{1}(L(3,1),\lambda)$. Take a sequence of strictly convex contact forms $\lambda_{n}$ such that $\lambda_{n} \to \lambda$ in $C^{\infty}$-topology and $\lambda_{n}$ is non-degenerate for each $n$. 
    Therefore we have
    \begin{equation}
        \inf_{\gamma\in \mathcal{S}_{3},\mu_{\disk}(\gamma^{3})=3}\int_{\gamma}\lambda_{n} = \frac{1}{3}\,c_{1}^{\mathrm{ECH}}(L(3,1),\lambda_{n})
    \end{equation}
    Note that  $c_{1}^{\mathrm{ECH}}(L(3,1),\lambda_{n})\to L$ as $n\to+\infty$. This means that there is a sequence of $\gamma_{n}\in \mathcal{S}_{3}(L(3,1),f_{n}\lambda)$ with $\mu_{\disk}(\gamma_{n}^{3})=3$ such that $\int_{\gamma_{n}}\lambda_{n}\to \frac{1}{3}L$. By Arzelà–Ascoli theorem, we can find a subsequence which converges to a periodic orbit $\gamma$ of $\lambda$  in $C^{\infty}$-topology. 
    \begin{cla}\label{limittingorbit}
       $\gamma$ is simple. In particular, $\gamma\in \mathcal{S}_{3}(L(3,1),\lambda)$ and $\mu_{\disk}(\gamma^{3})=3$.
    \end{cla}
    \begin{proof}[\bf Proof of Claim \ref{limittingorbit}]
        By the argument so far, there is a sequence of $\gamma_{n}\in \mathcal{S}_{3}(L(3,1),\lambda_{n})$ with $\mu_{\disk}(\gamma_{n}^{3})=3$ which converges to $\gamma$ in $C^{\infty}$. Suppose that $\gamma$ is not simple, that is, there is a simple orbit $\gamma'$ and $k\in \mathbb{Z}_{>0}$ with $\gamma'^{k}=\gamma$. From the lower semi-continuity of $\mu$, we have $\mu_{\disk}(\gamma_{n}^{3})\to\mu_{\disk}(\gamma'^{3k})=\mu_{\disk}((\gamma'^{3})^{k})=3$.  This contradicts Proposition \ref{conleycovering}. Therefore $\gamma$ is simple. This means that for sufficiently large $n$, $\gamma_{n}$ is transversally isotopic to $\gamma$.  Therefore, $\gamma$ is 3-unknotted and  has self-linking number $-\frac{1}{3}$.
        
        At last, we prove $\mu_{\disk}(\gamma^3)=3$. From the lower semi-continuity of $\mu$, we have $\mu_{\disk}(\gamma_{n}^3)\to\mu_{\disk}(\gamma^3)=3$ or $2$. $\mu_{\disk}(\gamma^3)=2$ contradicts the assumption of dynamical convexity. Thus we have $\mu_{\disk}(\gamma^3)=3$. We complete the proof.
    \end{proof}
    As discussion so far, there is a sequence of $\gamma_{n}\in \mathcal{S}_{3}(L(3,1),f_{n}\lambda)$ with $\mu_{\disk}(\gamma_{n}^3)=3$ and $\gamma\in \mathcal{S}_{3}(L(3,1),\lambda)$ with $\mu_{\disk}(\gamma^3)=3$ such that $\int_{\gamma_{n}}f_{n}\lambda\to \frac{1}{3}L$ and $\gamma_{n}$ converges to  $\gamma$ of $\lambda$  in $C^{\infty}$-topology. Therefore we have $\int_{\gamma}\lambda = \frac{1}{3}\,c_{1}^{\mathrm{ECH}}(L(3,1),\lambda)$ in $C^{\infty}$-topology. we complete the proof of Proposition \ref{degeneratel2}.
\end{proof}

Having Proposition \ref{degeneratel2}, to complete the proof of Theorem \ref{maintheorem}, it is sufficient to show the next proposition.
\begin{prp}\label{upperbound}
     Assume that $(L(3,1),\lambda)$ is strictly convex. Then
    \begin{equation}
   \frac{1}{3}\,c_{1}^{\mathrm{ECH}}(L(3,1),\lambda) \leq \inf_{\gamma\in \mathcal{S}_{3},\mu_{\disk}(\gamma^3)=3}\int_{\gamma}\lambda .
\end{equation}
\end{prp}
\begin{proof}[\bf Proof of Proposition \ref{upperbound}]
    We prove this by contradiction. Suppose that there exists $\gamma_{\lambda}\in \mathcal{S}_{3}(L(3,1),\lambda)$ with $\mu_{\disk}(\gamma_{\lambda}^3)=3$ such that $\frac{1}{3}\,c_{1}^{\mathrm{ECH}}(L(3,1),\lambda)>\int_{\gamma_{\lambda}}\lambda$.

    \begin{lem}\label{seqsmooth}
         There exists a sequence of smooth functions $f_{n}:L(3,1)\to \mathbb{R}_{>0}$  such that $f_{n}\to 1$ in $C^{\infty}$-topology and satisfying $f_{n}|_{\gamma_{\lambda}}=1$ and $df_{n}|_{\gamma_{\lambda}}=0$. Moreover, all periodic orbits of $X_{f_{n}\lambda}$ of periods $<n$ are non-degenerate and all contractible orbits of periods $<n$ have Conley-Zehnder index $\geq 3$. In addition, $\gamma_{\lambda}$ is a non-degenerate periodic orbit of $X_{f_{n}\lambda}$ with $\mu_{\disk}(\gamma_{\lambda}^3)=3$ for every $n$.
    \end{lem}
    \begin{proof}[\bf Proof of Lemma \ref{seqsmooth}]
        See \cite[Lemma 6.8, 6.9]{HWZ4}
    \end{proof}

For a sequence of smooth functions $f_{n}:L(3,1)\to \mathbb{R}_{>0}$ in Lemma \ref{seqsmooth}, fix $N>>0$ sufficient large so that $c_{1}^{\mathrm{ECH}}(L(3,1),f_{N}\lambda)>\int_{\gamma_{\lambda}}\lambda$ and $N>2c_{1}^{\mathrm{ECH}}(L(3,1),f_{N}\lambda)$. We may take such $f_{N}$ because $c_{1}^{\mathrm{ECH}}$ is continuous in $C^{0}$-topology.

\begin{lem}\label{riginalchange}
    Let $f:L(3,1)\to \mathbb{R}_{>0}$ be a smooth function such that $f(x)<f_{N}(x)$ for any $x\in L(3,1)$. Suppose that $f\lambda$ is non-degenerate dynamically convex. Then there exists a simple periodic orbit $\gamma \in \mathcal{S}_{3}(L(3,1),f\lambda)$ 
with $\mu(\gamma)=1$ such that $\int_{\gamma}f\lambda<\int_{\gamma_{\lambda}}\lambda$.

\end{lem}
\begin{proof}[\bf Outline of the proof of 
Lemma \ref{riginalchange}]
See \cite[Proposition 3.1]{HrS}. In the proof and statement of \cite[Proposition 3.1]{HrS}, ellipsoids are used instead of $(L(3,1),f_{N}\lambda)$, but the important point in the proof is to find $3$-unknotted self-linking number $-\frac{1}{3}$ orbit $\gamma$ with  $\mu_{\disk}(\gamma^3)=3$ and construct a suitable $J$-holomorphic curve from \cite[Proposition 6.8]{HrLS}. Now, we have $\gamma_{\lambda} \in \mathcal{S}_{3}(L(3,1),f_{N}\lambda)$ 
with $\mu_{\disk}(\gamma_{\lambda}^3)=3$ and hence  by applying \cite[Proposition 6.8]{HrLS}, we can construct a suitable $J$-holomorphic curve. By using this curves instead of ones in the original proof, we can show Proposition \ref{riginalchange}. Here we note that the discussion in the proof of Lemma \ref{lem:boundaryop} is needed to prove the same result of \cite[Theorem 3.15]{HrS}
\end{proof}
Now, we would complete the proof of Proposition.  Let $f:L(3,1)\to \mathbb{R}_{>0}$ be a smooth function such that $f(x)<f_{N}(x)$ for any $x\in L(3,1)$, $f\lambda$ be non-degenerate strictly convex and $\int_{\gamma_{\lambda}}\lambda<c_{1}^{\mathrm{ECH}}(L(3,1),f\lambda)<c_{1}^{\mathrm{ECH}}(L(3,1),f_{N}\lambda)$. We can check easily that it is possible to take such $f$.  Due to Lemma \ref{riginalchange}, there exists a simple periodic orbit $\gamma \in \mathcal{S}_{3}(L(3,1),f\lambda)$ 
with $\mu_{\disk}(\gamma^3)=3$ such that $\int_{\gamma}f\lambda<\int_{\gamma_{\lambda}}\lambda$. Since  $\inf_{\gamma\in \mathcal{S}_{3},\mu_{\disk}(\gamma^3)=3}\int_{\gamma}f\lambda = \frac{1}{3}\,c_{1}^{\mathrm{ECH}}(L(3,1),f\lambda)$, we have $\int_{\gamma_{\lambda}}\lambda<c_{1}^{\mathrm{ECH}}(L(3,1),f\lambda)\leq \int_{\gamma}f\lambda$. This is a contradiction. We complete the proof.
\end{proof}

\end{document}